\documentclass[a4paper,12pt,onecolumn]{article}

\usepackage[vmargin=2cm,hmargin=2cm,headheight=14.5pt,top=2cm,headsep=.5cm]{geometry}
\usepackage[utf8]{inputenc}
\usepackage{bm}
\usepackage{empheq}
\usepackage{stackrel}
\usepackage{cases}
\usepackage{mathtools,frame}
\usepackage{amsthm,amsmath,amscd}
\usepackage{makeidx,stmaryrd,stackengine}\stackMath
\usepackage[charter]{mathdesign}
\usepackage{tikz-cd}
\usepackage{caption}
\usepackage{xcolor}
\tikzcdset{every label/.append style = {font = \small}}

\usepackage{graphicx}
\DeclareMathSizes{12}{12}{8}{6}

\usepackage{cite}
\usepackage{url}
\usepackage[ddmmyy]{datetime}

\usepackage{marvosym}

\newtheoremstyle{ptheorem}{1em}{0em}{\itshape}{}{\bfseries}{.}{.5em}{\thmname{#1}\thmnumber{ #2}\thmnote{ (\hspace{-.01pt}{#3})}}

\theoremstyle{ptheorem}

\newtheorem{thm}{Theorem}[section]
\newtheorem{pro}[thm]{Proposition}
\newtheorem{lem}[thm]{Lemma}
\newtheorem{cor}[thm]{Corollary}

\newtheoremstyle{hdef}{1em}{0em}{}{}{\bfseries}{.}{.5em}{\thmname{#1}\thmnumber{ #2}\thmnote{ (\hspace{-.01pt}{#3})}}
\theoremstyle{hdef}

\newtheorem{dfn}[thm]{Definition}
\newtheorem{rem}[thm]{Remark}
\newtheorem{exa}[thm]{Example}

\newtheoremstyle{obs}{1em}{0em}{\color{blue}}{}{\bfseries}{.}{.5em}{\thmname{#1}\thmnumber{ #2}\thmnote{ (\hspace{-.01pt}{#3})}}
\theoremstyle{obs}

\makeatletter
\newtheoremstyle{premark}{1em}{0em}{
\addtolength{\@totalleftmargin}{1.5em}
\addtolength{\linewidth}{-1.5em}
\parshape 1 1.5em \linewidth}{}{\scshape}{.}{.5em}{}
\makeatother

\theoremstyle{premark}

\numberwithin{equation}{section}
\numberwithin{figure}{section}

\DeclareMathOperator{\Id}{Id}

\DeclareMathOperator{\dif}{d}

\newcommand{\cB}{{\mathcal B}}
\newcommand{\cC}{{\mathcal C}}

\newcommand{\cE}{{\mathcal E}}

\newcommand{\cM}{{\mathcal M}}

\newcommand{\bN}{{\mathbb N}}

\newcommand{\bQ}{{\mathbb Q}}
\newcommand{\bR}{{\mathbb R}}
\newcommand{\bS}{{\mathbb S}}

\renewcommand{\a}{\alpha}
\renewcommand{\b}{\beta}
\renewcommand{\c}{\gamma}

\newcommand{\e}{\varepsilon}

\renewcommand{\phi}{\varphi}

\newcommand{\ol}{\overline}
\newcommand{\fa}{\forall}

\newcommand{\n}{{n\in\bN}}

\newcommand{\Ra}{\Rightarrow}

\newcommand{\nkp}{\enskip}

\newcommand{\sfa}{\nkp\fa}
\renewcommand{\d}{\delta}
\newcommand{\D}{\Delta}

\renewcommand{\(}{\left(}
\renewcommand{\)}{\right)}

\newcommand{\lil}{\lim\limits}

\newcommand{\lili}[1]{\stackrel[{#1}]{}{\underline\lim}}
\newcommand{\til}{\widetilde}

\AtBeginDocument{
	
}

\allowdisplaybreaks

\title{Displacement Calculus\footnote{The authors were partially supported by Ministerio de Econom\'ia y Competitividad, Spain, and FEDER, project MTM2013-43014-P, and by the Agencia Estatal de Investigaci\'on (AEI) of Spain under grant MTM2016-75140-P, co-financed by the European Community fund FEDER.}}

\author{
Ignacio Márquez Albés\footnote{Partially supported by Xunta de Galicia, grant ED481A-2017/095.}\\
\normalsize e-mail: ignacio.marquez@usc.es\\
F. Adri\'an F. Tojo \\
\normalsize e-mail: fernandoadrian.fernandez@usc.es\\
\normalsize \emph{Instituto de Ma\-te\-m\'a\-ti\-cas, Facultade de Matem\'aticas,} \\ \normalsize\emph{Universidade de Santiago de Com\-pos\-te\-la, Spain.}\\ 
}
\date{}

\allowdisplaybreaks
\begin{document}
 \maketitle




\begin{abstract}

In this work we establish a theory of Calculus based on the new concept of \emph{displacement}. We develop all the concepts and results necessary to go from the definition to differential equations, starting with topology and measure and moving on to differentiation and integration. We find interesting notions on the way, such as the integral with respect to a path of measures or the displacement derivative. We relate both of these two concepts by a Fundamental Theorem of Calculus. Finally, we develop the necessary framework in order to study displacement equations by relating them to Stieltjes differential equations.
\end{abstract}

\medbreak


\medbreak


\section{Introduction}

Derivatives are, in the classical sense of Newton \cite{Newton}, infinitesimal rates of change of one (dependent) variable with respect to another (independent) variable. Formally, the derivative of $f$ with respect to $x$ is
\[f'(x):=\lim_{\Delta x\to 0}\frac{\Delta f}{\Delta x}.\]
 The symbol $\Delta$ represents what we call the \emph{variation}, that is, the change of magnitude underwent by a given variable\footnote{In the \emph{calculus of variations} the study variable is a function $f$ and the variation of $f$ is noted by $\delta f$.}. This variation is, in the classical setting, defined in the most simple possible way as $\Delta x=\til x-x$, where $x$ is the point at which we want to compute the derivative (the point of departure) and $\til x$ another point which we assume close enough to $x$. From this, it follows naturally that the variation of the dependent variable has to be expressed as $\Delta f=f(\til x)-f(x)$. This way, when $\til x$ tends to $x$, that is, when $\Delta x$ tends to zero, we have
\[f'(x):=\lim_{\til x\to x}\frac{f(\til x)-f(x)}{\til x-x}.\]
Of course, this naïve way of defining the variation is by no means the unique way of giving meaning to such expression. The intuitive idea of variation is naturally linked to the mathematical concept of distance. After all, in order to measure how much a quantity has varied it is enough to see \emph{how far apart} the new point $\til x$ is from the first $x$ that is, we have to measure, in some sense, the distance between them. This manner of extending the notion of variation --and thus of derivative-- has been accomplished in different ways. The most crude of these if what is called the \emph{absolute derivative}.
\begin{dfn}[{\cite[expression (1)]{ChIn}}] Let $(X,d_X)$ and $(Y,d_Y)$ be two metric spaces and consider $f:X\to Y$ and $x\in X$. We say $f$ is \emph{absolutely differentiable at $x$} if and only if the following limit --called absolute derivative of $f$ at $x$-- exists:
\[f^{|\prime|}(x):=\lim_{\til x\to x}\frac{d_Y(f(x),f(\til x))}{d_X(x,\til x)}.\]
\end{dfn}
In the case of differentiable functions $f:\bR\to\bR$ we have that, as expected, $f^{|\prime|}=|f'|$ \cite[Proposition 3.1]{ChIn}. Hence, this result conveys the true meaning of the absolute derivative --it is the absolute value of the derivative-- and it extends the notion of derivative to the broader setting of metric spaces. Even so, this definition may seem somewhat unfulfilling as a generalization. For instance, in the case of the real line, it does not preserve the spirit of the intuitive notion of \emph{`infinitesimal rates of change'}: changes of rate have, of necessity, to be allowed to be \emph{negative}.

A more subtle extension of differentiability to the realm of metric spaces can be achieved through \emph{mutational analysis} where the affine structure of differentials is changed by a family of functions, called \emph{mutations}, that mimic the properties and behavior of derivatives. We refer the reader to \cite{Lorenz} for more information on the subject.

The considerations above bring us to another possible extension of the notion of derivative: that of the \emph{Stieltjes derivative}, also known as \emph{$g$--derivative}. Here we present the definition used in \cite{PoRo}. However, a lot of previous work exists on the topic of differentiation with respect to a function, such as the work Averna and Preiss, \cite{Aversa1999}, Daniell \cite{Daniell1918,Daniell1929} or even more classical references like 
	\cite{Lebesgue2009}.

\begin{dfn}[\cite{PoRo}]Let $g:\bR\to\bR$ be a monotone nondecreasing function which is continuous from the left. The \emph{Stieltjes derivative with respect to $g$} --or \emph{$g$--derivative}-- of a function $f:\bR\to\bR$ at a point $x\in\bR$ is defined as follows, provided that the corresponding limits exist:
\begin{align*} f'_g(x) & =\lim_{\til x\to x}\frac{f(\til x)-f(x)}{g(\til x)-g(x)} \text{ if } g \text{ is continuous at } x \text{, or}\\
 f'_g(x) & =\lim_{\til x\to x^+}\frac{f(\til x)-f(x)}{g(\til x)-g(x)} \text{ if } g \text{ is discontinuous at } x.
 \end{align*}
\end{dfn}
Clearly, we have defined $\Delta x$ through a rescaling of the abscissae axis by $g$. Observe that, although $d(x, \til x)=|g(\til x)-g(x)|$ is a pseudometric \cite{FP2016}, $\Delta x=g(\til x)-g(x)$ is allowed to change sign. 

The aim of this paper is to take this generalization one step further in the following sense. The definition of $\Delta x$ does not have to depend on a rescaling, but its absolute value definitely has to suggest, in a broad sense, the notion, if not of distance, of being \emph{far apart} or \emph{close} as well as the \emph{direction} --change of sign. That is why we introduce the notion of \emph{displacement} (Definition~\ref{deltadef}). This definition takes to full generality the ideas and results in \cite{PoRo,FP2016}.

This work is structured as follows. In Section~\ref{displacements} we define the basic concept the rest of the paper revolves around: the notion of \emph{displacement space}. Specifically, in Subsection~\ref{definition} we develop the definition and basic properties of displacements, linking them to previously known concepts and illustrating their diversity with several examples. On the other hand, in Subsection~\ref{topology} we endow the displacement space with a natural topology and prove various useful properties. 

Section~\ref{measure} deals with the construction of a measure associated to displacement spaces. We restrict ourselves to the real line, where we first construct its associated measure as a Lebesgue-Stieltjes measure. Then, we construct a theory of integration for displacement spaces. Here we define the concept of \emph{integral with respect to a path of measures} which will be the key to defining an integral associated to a displacement.

Section~\ref{derivatives} is devoted to the definition and properties of a displacement derivative which will be later be proven to be compatible with the displacement measure in that we can provide a Fundamental Theorem of Calculus relating both of them (Theorems~\ref{FTC} and~\ref{FTC2}). Later, in Section~\ref{Stieltjes}, we study the connection existing between this type of derivatives and Stieltjes derivatives and, in Section 6, we propose a diffusion model on smart surfaces based on displacements.

The last section is devoted to the conclusions of this work and the open problems lying ahead.

\section{Displacement spaces}\label{displacements}
In this section we focus on the definition of displacement spaces. This new framework is then illustrated with some examples which show, for example, that every set equipped with a metric map is a displacement space. We also study a topological structure that displacement spaces can be endowed with.
\subsection{Definitions and properties}\label{definition}
Let us make explicit the basic definition of this paper.
\begin{dfn}\label{deltadef} Let $X\ne\emptyset$ be a set. A \emph{displacement} is a function $\Delta:X^2\to\bR$ such that the following properties hold:
\begin{itemize}
	\item[(H1)] $\Delta(x,x)=0,\ x\in X$.
	\item[(H2)] For all $x,y\in X$,
	\[\lili{z\rightharpoonup y}|\Delta(x,z)|=|\Delta(x,y)|,\]
	where
	\[\lili{z\rightharpoonup y}|\Delta(x,z)|:=\sup \left\{\liminf_{n\to\infty}|\Delta(x,z_n)|: (z_n)_{n\in\bN}\subset X,\ \Delta(y,z_n)\xrightarrow{n\to\infty}0\right\}.\]
	\end{itemize}
	All limits occurring in this work will be considered with the usual topology of $\bR$. A pair $(X,\Delta)$ is called a \emph{displacement space}.
\end{dfn}

\begin{rem}Why (H1) and (H2)? These two hypotheses are of prominent topological flavor. (H1) will guarantee that open balls are nonempty in the to-be-defined non-necessarily-metric topology related to $\Delta$. On the other hand, (H2) will be sufficient (and indeed necessary) to show that open balls are, indeed, open (Lemma~\ref{lemob}) and that the $\Delta$--topology is second countable (Lemma~\ref{count}). We will later discuss (Remark~\ref{bigrem}) whether or not we can forestall (H2) when we restrict ourselves to displacement calculus.
\end{rem}
\begin{rem}
	Note that, for (H2) to be satisfied, it is enough to show that $\lili{z\rightharpoonup y}|\Delta(x,z)|\le|\Delta(x,y)|$ for all $x,y\in X$, as the reverse inequality always holds.
\end{rem}

The following lemma gives a useful sufficient condition for (H2) to be satisfied.
\begin{lem}\label{ti}
	Let $X$ be a set and $\Delta:X^2\to\bR$. Assume that the following property holds:
	\begin{enumerate}
		\item[\textup{(H2')}] There exists a strictly increasing left--continuous map $\varphi:[0,+\infty)\rightarrow[0,+\infty)$, continuous at $0$, satisfying $\varphi(0)=0$ and such that, for $\psi(x,y):=\varphi(|\Delta(x,y)|)$,
		\begin{equation}\label{ppc}
		\psi(x,z)\leq \psi(x,y)+\psi(y,z);\quad x,y,z\in X.
		\end{equation}
	\end{enumerate}
	Then $\Delta$ satisfies \rm{(H2)}.
\end{lem} 
 \begin{proof}
 Fix $x,y\in X$ and let $(z_n)_{n\in \bN}\subset X$ such that $\lim_{n\to\infty}|\Delta(y,z_n)|= 0$. Then, condition~\eqref{ppc} yields
	$\psi(x,z_n)-\psi(y,z_n)\le \psi(x,y).$
	Hence,
	\[\psi(x,y)\ge \liminf_{n\to\infty}(\psi(x,z_n)-\psi(y,z_n))\ge \liminf_{n\to\infty}\psi(x,z_n)-\limsup_{n\to\infty}\psi(y,z_n).\]
	Since $\varphi(0)=0$, $\varphi$ is continuous at $0$ and $\lim_{n\to\infty}|\Delta(y,z_n)|= 0$, $\limsup_{n\to\infty}\psi(y,z_n)=0$, so
	\begin{equation}\label{varin}
		\psi(x,y)\ge \liminf_{n\to\infty}\psi(x,z_n).
	\end{equation}
 Let us show that 
 \begin{equation}\label{varin2}
 		\liminf_{n\to\infty}\phi(|\Delta(x,z_n)|)\ge\phi\(\liminf_{n\to\infty}|\Delta(x,z_n)|\).
 \end{equation}
 Indeed, by definition of $\liminf$, we have that for any $\e\in\bR^+$ there exists $n_0\in\bN$ such that if $n\ge n_0$ then
 $|\Delta(x,z_n)|\ge\liminf_{n\to\infty}|\Delta(x,z_n)|-\e.$ 
 Since $\phi$ is strictly increasing, for each $\e\in\bR^+$ there exists $n_0\in\bN$ such that for $n\ge n_0$ we have
 \[\phi\(|\Delta(x,z_n)|\)\ge\phi\(\liminf_{n\to\infty}|\Delta(x,z_n)|-\e\).\]
 Thus, for any $\e>0$ we have that
 \[\liminf_{n\to\infty}\phi\(|\Delta(x,z_n)|\)\ge\phi\(\liminf_{n\to\infty}|\Delta(x,z_n)|-\e\),\]
 which, using the left--continuity of $\phi$, leads to~\eqref{varin2}.
 Hence, it follows from~\eqref{varin} and~\eqref{varin2} that
 \[\phi(|\Delta(x,y)|)=\psi(x,y)\ge \liminf_{n\to\infty}\psi(x,z_n)=\liminf_{n\to\infty}\phi(|\Delta(x,z_n)|)\ge\phi\(\liminf_{n\to\infty}|\Delta(x,z_n)|\),\]
 which, together with the fact that $\phi$ is strictly increasing, yields that 
 \[|\Delta(x,y)|\ge\liminf_{n\to\infty}|\Delta(x,z_n)|.\]
	Since this holds for any $(z_n)_{n\in \bN}$ in $X$ such that $\lim_{n\to\infty}|\Delta(y,z_n)|= 0$, we have the desired result.
\end{proof}

Lemma~\ref{ti} illustrates that condition (H2) is a way of avoiding the triangle inequality --or more general versions of it-- which is common to metrics and analogous objects. We can find similar conditions in the literature. For instance, in \cite[Definition 3.1]{Roldan}, they use, while defining an \emph{RS--generalized metric space} $(X,\Delta)$, the condition
\begin{enumerate}
		\item[\textup{($D_3'$)}] There exists $C>0$ such that if $x,y\in X$ and \[\lim_{n\to\infty}\Delta(x_n,x)=\lim_{n\to\infty}\Delta(x,x_n)=\lim_{n,m\to+\infty}\Delta(x_n,x_m)=0,\] then
		\[\Delta(x,y)\le C\limsup\Delta(x_n,y).\]
	\end{enumerate}
More complicated conditions can be found in \cite[(H3) Section~3.1, (H3') Section~4.1]{Lorenz}.

Last, we remark that the same statement as (H2'), but dropping the left-continuity, is actually sufficient to prove the results in this work.


In the next examples we use the sufficient condition provided by Lemma~\ref{ti}.
\begin{exa}\label{exa1}
	Consider the sphere $\bS^1$ and define the following map:
\begin{center}
	\begin{tikzcd}[row sep=tiny]
		\bS^1\times\bS^1 \arrow{r}{\Delta} & {[}0,2\pi)\\
		(x,y) \arrow[mapsto]{r} & \min\{\theta\in[0+\infty)\ :\ xe^{i\theta}=y\}.
	\end{tikzcd}\end{center}
$\Delta(x,y)$ is a displacement that measures the minimum counter-clockwise angle necessary to move from $x$ to $y$. From a real life point of view, this map describes the way cars move in a roundabout. Suppose that a car enters the roundabout at a point $x$ and wants to exit at a point $y$. In that case, circulation rules force the car to move in a given direction, which happens to be counter-clockwise in most of the countries around the world. In this case, drivers are assumed to take the exit $y$ as soon as they reach it.

It is clear that (H1) holds. For (H2'), take $\phi(r)=r$. Then, for $x,y,z\in\bS^1$, if $\Delta(x,y)+\Delta(y,z)\ge 2\pi$, then (H2') clearly holds. Otherwise,
$\Delta(x,z)=\Delta(x,y)+\Delta(y,z)$, so (H2') holds.
\end{exa}
\begin{exa}
	Let $(X,E)$ be a complete weighted directed graph, that is, $X=\{x_1,\dots,x_n\}$ is a finite set of $n\in\bN$ vertices and $E\in\cM_n(\bR)$ is a matrix with zeros in the diagonal and positive numbers everywhere else. The element $e_{j,k}$ of the matrix $E$ denotes the weight of the directed edge from vertex $x_j$ to vertex $x_k$. This kind of graph can represent, for instance, the time it takes to get from one point in a city to another by car, as Figure~\ref{fig:mapsantiago} illustrates.
\begin{figure}[h]
\centering
\includegraphics[width=.8\linewidth]{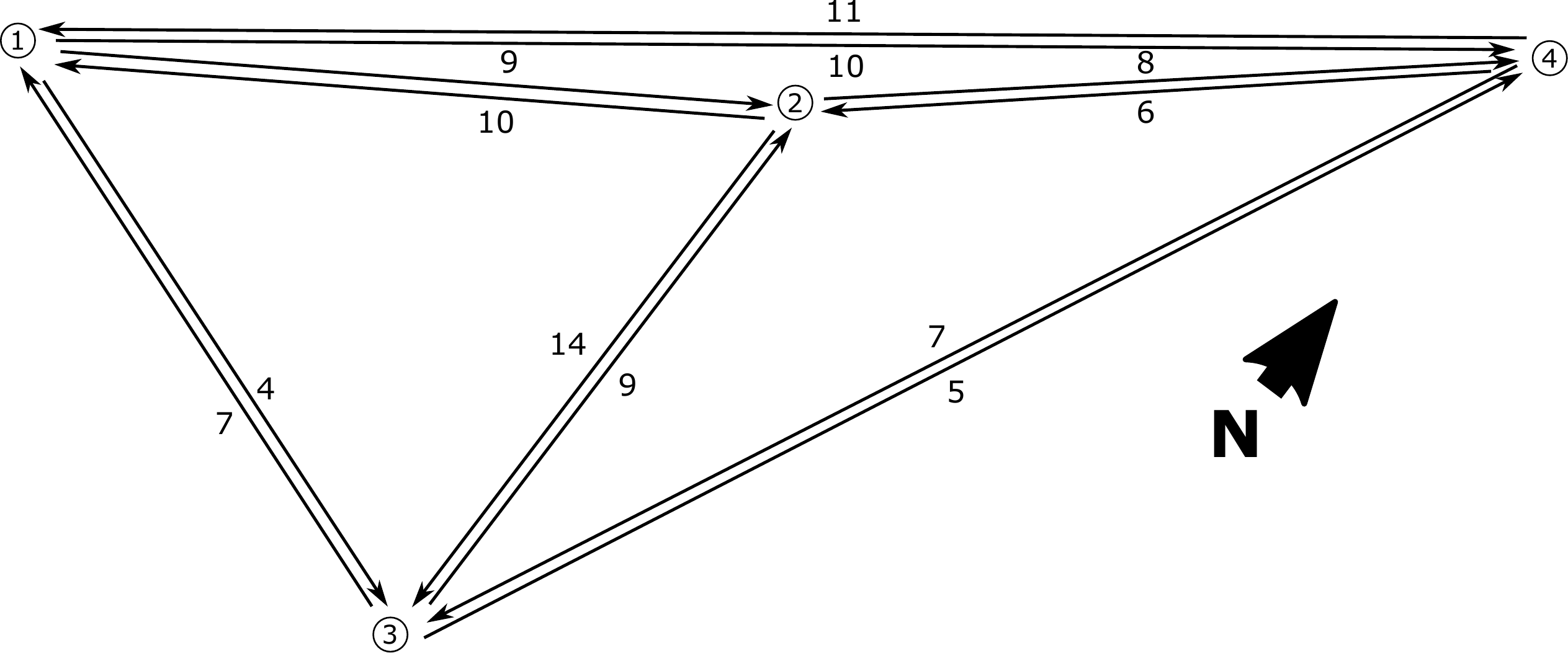}
\caption{Graph indicating the time in minutes it takes to go from one place to another in Santiago de Compostela by car (using the least time consuming path) according to \emph{Google Maps} --good traffic conditions assumed. The points are placed in their actual relative geometric positions, being \textbf{1}: Faculty of Mathematics (USC), \textbf{2}: Cathedral, \textbf{3}: Train station, \textbf{4}: Bus station. Most of the streets in Santiago are one way, which accounts for the differences in time depending on the direction of the displacement.}
\label{fig:mapsantiago}
\end{figure}

Now, consider the set $\{x_1,\dots,x_4\}$ and the matrix $E$ as given in Figure~\ref{fig:mapsantiago}, that is,
\[E\equiv(e_{j,k})_{j,k=1}^4:=
\begin{pmatrix}
0 & 9 & 4 & 10 \\
10 & 0 & 14 & 8 \\
7 & 9 & 0 & 5 \\
11 & 6 & 7 & 0 \\
\end{pmatrix},\]
and the map $\Delta(x_j,x_k):=e_{j,k}$. It can be checked that $\Delta$ is subadditive --which is to be expected since, if we could get faster from a point to another through a third one \emph{Google Maps} would have chosen that option. Hence, (H2') holds for $\phi(r)=r$, and so $\Delta$ is a displacement.
\end{exa}
\begin{exa}[Zermelo's navigation problem]
	 In 1931, Zermelo solved the following navigation problem \cite{Zer}. Let $F=(u,v)\in\cC(\bR^2,\bR^2)$ be a vector field, for instance, the velocity field of the wind on top of a body of water, or the velocity field of the water itself. Assume an object that moves with constant celerity $V$ on that body of water wants to go from a point $A$ (which we can assume at the origin) to a point $B$. Which is the least time consuming path to take?

We are going to assume that $V>W:=\max_{x\in\bR^2}\sqrt{u(x)^2+v(x)^2}$, that is, the object can navigate against wind. Zermelo proved, using variational methods, that the solution of the problem satisfies the following system of partial differential equations:

\begin{align*}
 x' &= V\cos \theta + u, \\
 y' &= V\sin\theta + v, \\
 \theta' &
 = \sin^2\theta \frac{\partial v}{\partial x} 
 + \sin\theta \cos\theta \left(\frac{\partial u}{\partial x} - \frac{\partial v}{\partial y}
 \right) - \cos^2\theta\frac{\partial u}{\partial y},	
\end{align*}
being the last equation known as \emph{Zermelo's equation}. Observe that, if $u,v\in\cC^2(\bR^2)$, there exists a unique solution of the system. Through the change of variables $(\til x,\til y)=B-(x,y)$, instead of going from the origin to the point $B$ we go from $B$ to the origin, and the equations will provide a different time. This illustrates the fact that, when measuring how far apart something is in terms of time, symmetry is not generally satisfied. For instance, if we measure the distance between two points of a river by the time it takes to get from one point to another it is not the same to go upstream than downstream.

Let $A,B\in\bR^2$. If $\Delta(A,B)$ is the smallest time necessary to arrive from $A$ to $B$ in Zermelo's navigation problem, $\Delta\ge 0$ is a displacement on $\bR^2$, for $\Delta$ is subadditive and (H1)--(H2') are clearly satisfied.

In the symmetric setting --that is, $\Delta(x,y)=\Delta(y,x)$-- this problem is a paradigmatic example of \emph{Finslerian length space}. The theory regarding these spaces has been thoroughly developed but, as stated in \cite{Burago}, although \emph{``one could modify the definitions to allow non symmetric length structures and metrics''}, this case has not been studied yet. What we present in this paper might be an starting point for a theory of \emph{non symmetric length spaces}.
\end{exa}
\begin{exa}\label{Stdisp}
The map $\Delta:\bR\times \bR\rightarrow \bR$ defined as $\Delta(x,y)=g(y)-g(x)$ for a nondecreasing and left-continuous function $g$ (cf.\cite{PoRo}) is a displacement as it satisfies (H2') for $\phi=\Id$. In what follows, we will refer to these displacements as \emph{Stieltjes displacements}. Furthermore, the following lemma shows a way to ensure that a displacement is a Stieltjes displacement.
\end{exa}
\begin{lem} Let $(X,\Delta)$ be a displacement space. Then there exists $g:X\to\bR$ such that $\Delta(x,y)=g(y)-g(x)$ for every $x,y\in X$ if and only if, for every $x,y,z\in X$,
\begin{enumerate}
\item $\Delta(x,y)=-\Delta(y,x)$,
\item $\Delta(x,z)=\Delta(x,y)+\Delta(y,z)$.
\end{enumerate}
\end{lem}
\begin{proof}Necessity is straightforward. In order to prove sufficiency, take $x_0\in X$ and define $g(x)=\Delta(x_0,x)$ for $x\in X$. Then,
\[\Delta(x,y)=\Delta(x,x_0)+\Delta(x_0,y)=-\Delta(x_0,x)+g(y)=g(y)-g(x).\qedhere\]
\end{proof}

For the interest of the work ahead, we include the following example of a non--Stieltjes displacement.

\begin{exa}\label{dispexp}
	Let $X=[0,1]$ and $\Delta:X\times X\to[0,+\infty)$ given by
	\[\Delta(x,y)=e^{y^2-x^2}-e^{x-y}.\]
	Clearly, condition (H1) is satisfied. Now, for condition (H2), fix $x,y\in X$. Since $\Delta (x,\cdot)$ is a continuous function, it is enough to show that
	\begin{equation}\label{condejexp}
	|\Delta (x,y)|=\sup\left\{\lim_{n\to\infty} |\Delta(x,z_n)|: (z_n)_{n\in\mathbb N}\subset X,\, \lim_{n\to\infty} |\Delta(y,z_n)|=0 \right\}.
	\end{equation}
	Let $(z_n)_{n\in\mathbb N}\subset X$ be a sequence such that $|\Delta(y,z_n)|\xrightarrow{n\to\infty}0$.
	Then, we have that $e^{z_n^2+z_n}\xrightarrow{n\to\infty}e^{y^2-y}$, from which we get that
	\begin{equation}\label{limexpej}
	\lim_{n\to\infty} z_n^2+z_n=y^2+y.
	\end{equation}
	Define
	\[h(t)=\frac{-1+\sqrt{1+4t}}{2},\quad t\in[0,+\infty).\]
	For any $a\in[0,1]$ we have that $h(a^2+a)=a$. Therefore, applying $h$ to both sides of~\eqref{limexpej} and noting that $h$ is a continuous function, we obtain
	\[y=h(y^2+y)=h\left(\lim_{n\to\infty} z_n^2+z_n\right)=\lim_{n\to\infty}h(z_n^2+z_n)=\lim_{n\to\infty} z_n.\]
	That is, if $(z_n)_{n\in\mathbb N}\subset X$ is such that $|\Delta(y,z_n)|\xrightarrow{n\to\infty}0$, it follows that $(z_n)\xrightarrow{n\to\infty} y$. Hence~\eqref{condejexp} is trivially satisfied, and thus $\Delta$ is a displacement. However, $\Delta$ is not a Stieltjes displacement as
	\[\Delta(1,0)=e-e^{-1}\not= e^{-1}-e=\Delta(0,1).\]
\end{exa}

\subsection{Displacement topologies}\label{topology}
It is a well--known result that a set equipped with a metric map generates a topology through the definition of open balls. The same thing happens with displacement spaces. However, fewer nice properties can be obtained from just the definition.
\begin{dfn}
Given a displacement space $(X,\Delta)$, $x\in X$ and $r\in\bR^+$, we define the \textit{$\Delta$--ball} or simply \emph{ball}) \emph{of center $x$ and radius $r$} as \[B_\Delta(x,r):=\{y\in X\ :\ |\Delta(x,y)|<r\}.\]
Also, we define the \emph{$\Delta$--topology} in the following way:
\[\tau_\Delta:=\left\{U\subset X\ :\ \fa x\in U\ \exists r\in\bR^+,\ B_\Delta(x,r)\subset U\right\}.\]
\end{dfn}

Clearly, $\tau_\Delta$ is a topology. We denote by $\cE_\Delta$ the set of $\Delta$--balls in $X$ and by $\tau_u$ the usual euclidean topology of $\bR^n$ for any $n\in\bN$.

We recall the following definition.

\begin{dfn} 
	Let $\tau_s$ be the topology generated by the intervals $(-\infty,r)\subset\bR$ and $(X,\tau)$ a topological space. We say $f:X\to\bR$ is \emph{upper-semicontinuous} if $f:(X,\tau)\to(\bR,\tau_s)$ is continuous.
\end{dfn}

In what follows we will write $\Delta_x(y):=\Delta(x,y)$.
\begin{lem}\label{lemob}
Let $\Delta:X \times X\to\bR$. Then the following are equivalent:
\begin{enumerate}
	\item $\Delta$ satisfies \textup{(H2)}.
	\item	$B_\Delta(x,r)\in\tau_\Delta$ for all $x\in X$ and $r\in\bR^+$.
	\item For all $x\in X$, $|\Delta_x|:X\to \bR$ is upper-semicontinuous.
	\end{enumerate}
\end{lem}
\begin{proof}
1$\Ra$2. Assume first that $(X,\Delta)$ satisfies (H2). Let $x\in X$ and $r\in\bR^+$ be fixed. If $B_\Delta(x,r)=\emptyset$ then $B_\Delta(x,r)\in\tau_\Delta$ trivially. Assume that $B_\Delta(x,r)\not=\emptyset$. Let us show that, for every $y\in B_\Delta(x,r)$, there exists $\e\in\bR^+$ such that $B_\Delta(y,\e)\subset B_\Delta(x,r)$. Assume this is not the case. Then, there exists $y\in B_\Delta(x,r)$ and $(z_n)_\n\subset X$ such that, for all $n\in\bN$,
\[|\Delta(y,z_n)|<1/n,\quad |\Delta(x,z_n)|\ge r. \]
Hence, we have a sequence $(z_n)_{n\in\bN}$ such that $\lim_{n\to\infty}|\Delta(y,z_n)|=0$ and, by (H2), 
\[\liminf_{n\to\infty}|\Delta(x,z_n)|\ge r>|\Delta(x,y)|=\lili{z\rightharpoonup y}|\Delta(x,z)|, \]
which contradicts the definition of supremum.

2$\Ra$1. Now, if $B_\Delta(x,r)\in \tau_\Delta$ for all $x\in X$ and $r\in\bR^+$, fix $x,y\in X$ and let $r:=|\Delta(x,y)|$, $\e\in\bR^+$. Clearly, $y\in B_\Delta(x,r+\e)$, so there exists $\d_\e\in\bR^+$ such that $B_\Delta(y,\d_\e)\subset B_\Delta(x,r+\e)$. Hence, if $(z_n)_\n\subset X$ is such that $|\Delta(y,z_n)|\to 0$ as $n\to\infty$, there exists $N\in \bN$ such that $|\Delta(y,z_n)|<\d_\e$ for every $n\ge N$, so $|\Delta(x,z_n)|<r+\e$ for every $n\ge N$. Hence $\liminf_{n\to\infty}|\Delta(x,z_n)|\le r+\e$. Since $\e$ was arbitrarily fixed, we get that $\liminf_{n\to\infty}|\Delta(x,z_n)|\le r$, which ends the result.

2$\Leftrightarrow$3. Just observe that $|\Delta_x|^{-1}((-\infty,r))=B_\D(x,r)$.
\end{proof}

\begin{rem}\label{h1r}
Note that hypothesis (H1) is not necessary for the previous result or the definition of the topology itself. In fact, it has only been used so far to show that open balls are nonempty. Allowing the open balls to be the empty set changes nothing as it always belongs to the topology, making the result true in any case.
However, hypothesis (H1) will be key in the definition of the displacement derivative in Section~\ref{derivatives}.
\end{rem}

\begin{lem}\label{lempd}
Let $\Delta:X \times X\to\bR$. If {\rm(H1)} and {\rm(H2)} hold then:
\begin{enumerate}
	\item Every element of $\cE_\Delta$ is nonempty.
	\item Every element of $\tau_\Delta$ is union of elements in $\cE_\Delta$.
	\item $\cE_\Delta$ is a basis of $\tau_\Delta$.

\end{enumerate}
\end{lem}
\begin{proof}
	1. Since $\Delta(x,x)=0$, $x\in B(x,r)$ for any $r\in\bR^+$.

	2. Fix $U\in\tau_\Delta$. By definition of $\tau_\Delta$, we know that for every $x\in U$ there exists $r_x\in\bR^+$ such that $B(x,r_x)\subset U$. Since $x\in B(x,r_x)$, we have that $X=\bigcup_{x\in X}B(x,r_x)$.

	3. By (H2) we have that $\cE_\Delta\subset\tau_\Delta$. If $x\in U\cap V$ for $U,V\in \cE_\Delta$, since $U$ and $V$ are open, so is $U\cap V$ and hence, using 1, there exists $W\in\cE_\Delta$ such that $x\in W\subset U\cap V$.
\end{proof}
\begin{exa}
	The conditions obtained in Lemma~\ref{lempd} do not suffice to obtain both (H1) and (H2). Consider the space $X=\{0,1\}$ together with de function $\Delta$ given by $\Delta(0,1)=\Delta(1,1)=0$, $\Delta(1,0)=\Delta(0,0)=1$. In this case $\cE_\Delta=\tau_\Delta$ and the topology coincides with that of the Sierpi\'nski space. Observe that $\Delta$ is not a displacement, although it satisfies the theses 1-3 of Lemma~\ref{lempd}.

	It is also worth to observe that the map $\til\Delta(0,1)=\til\Delta(1,1)=\til\Delta(0,0)=0$, $\til\Delta(1,0)=1$ is a displacement and $\tau_{\til \Delta}=\tau_\Delta$. This means that, if we want to find sufficient conditions in order for a $\Delta$ to be a displacement, those conditions cannot be purely topological. Furthermore, since the Sierpi\'nski space is not regular we deduce that it is not uniformizable, and thus we conclude that not every displacement space is uniformizable.

In the particular context of the real line, some further results can be obtained. In order to archieve them we ask for the following hypothesis for $\Delta:X^2\subset\bR^2\to\bR$.

	\begin{itemize}
		\item[\textup{(H3)}]$\Delta(x,y)\le\Delta(x,z)$ for every $x,y,z\in X$ such that $y\le z$.
		\end{itemize}
	\begin{lem}\label{count}
		Let $\Delta:\bR^2\to\bR$ satisfy {\rm(H1)--(H3)}. Then $(\mathbb R,\tau_\Delta)$ is a second-countable topological space. 
	\end{lem}

	\begin{proof}
		First of all, given $x\in\mathbb R$ and $r\in\mathbb R^+$, we can express $B_\Delta(x,r)$ as follows:
		\begin{displaymath}B_\Delta(x,r)=\{y\in\mathbb R: |\Delta(x,y)|<r\}=\{y\in\mathbb R: -r<\Delta_x(y)<r\}=\Delta_x^{-1}((-r,r)).\end{displaymath}
		Moreover, since $\Delta_x$ is non-decreasing, due to the bounded completeness of $(\bR,\le)$, $B_\Delta(x,r)$ is an interval (not necessarily open) with extremal points
		\begin{displaymath}a=\inf\{t\in\mathbb R: -r<\Delta_x(t)\},\quad b=\sup\{t\in\mathbb R: \Delta_x(t)<r\}.\end{displaymath}

		Let $U\in\tau_\Delta$. Then $U=\bigcup_{x\in U} B_\Delta(x,r_x)$ by definition of open set, and so, since each $B_\Delta(x,r_x)$ is a interval, we can write
		\begin{displaymath}U=\bigcup_{i\in\mathcal I}(a_i,b_i)\cup\bigcup_{j\in\mathcal J}[a_j,b_j)\cup\bigcup_{k\in\mathcal K}(a_k,b_k]\cup\bigcup_{l\in\mathcal L}[a_l,b_l],\end{displaymath}
		for some sets of indices $\mathcal I,\mathcal J,\mathcal K,\mathcal L$, where each of those intervals is an open ball of $\tau_\Delta$. 

		The set $A=\bigcup_{i\in\mathcal I}(a_i,b_i)$ is an open set in $(\bR,\tau_u)$ and therefore second countable, which implies that $A$ is Lindel\"of \cite[p. 182]{hart} and, hence, there exists a countable subcover of $A$, i.e., $A=\bigcup_{n\in\mathbb N}(a_{i_n},b_{i_n})$ for some set of indices $\{i_n\}_{n\in\bN}$. Similarly, the set $B=\bigcup_{j\in\mathcal J}[a_j,b_j)$ is an open set in the Sorgenfrey line, which is hereditarily Lindel\"of \cite[p. 79]{hart}, and so $B=\bigcup_{n\in\mathbb N}[a_{j_n},b_{j_n})$ for some set of indices $\{j_n\}_{n\in\bN}$. Analogously, the set $C=\bigcup_{k\in\mathcal K}(a_k,b_k]$ can be expressed as $\bigcup_{n\in\mathbb N}(a_{k_n},b_{k_n}]$. Finally, the set $D=\bigcup_{l\in\mathcal L}[a_l,b_l]$ can be decomposed as $D=\bigcup_{l\in\mathcal L}[a_l,b_l)\cup\bigcup_{l\in\mathcal L}(a_l,b_l]$, and once again, arguing as for the sets $B$ and $C$, we obtain that
		\[D=\bigcup_{l_n\in\bN}[a_{l_n},b_{l_n})\cup\bigcup_{l_n'\in\bN}(a_{l_n'},b_{l_n'}],\]
		for some sets of indices $\{l_n\}_{n\in\bN}$, $\{l_n'\}_{n\in\bN}$.
		However, by the definition of $D$ we have that $a_l,b_l\in D$ for all $l\in\mathcal L$, so
		\[D=\bigcup_{l_n\in\bN}[a_{l_n},b_{l_n}]\cup\bigcup_{l_n'\in\bN}[a_{l_n'},b_{l_n'}],\]
		which is clearly countable. Therefore, $U$ is the countable union of open balls, i.e., $\tau_\Delta$ is a second-countable topology.
	\end{proof}
	\begin{rem}\label{bigrem} This last proof relies heavily on the fact that the real number system, with its usual order, is bounded complete, that is, that every bounded (in the order sense) set has an infimum and a supremum. Observe also that the interaction between the topologies $\tau_u$ and $\tau_\Delta$ plays a mayor role in the proof. Finally, hypothesis (H2) is necessary in this result through Lemma~\ref{lemob}, which implies that open $\Delta$--balls are, indeed, open.

	Related to this last point, the authors would like to comment on the fact that hypothesis (H2) will not be necessary in the particular setting of the displacement calculus. However, it provides --as illustrated before with Lemmas~\ref{lemob} and~\ref{count}-- some information about the relation between $\tau_\Delta$ and the displacement calculus we are yet to develop. In particular, Lemma~\ref{count} shows that, for the real line, every $\tau_u$--Borel $\sigma$--algebra is, in particular, a $\tau_\Delta$--Borel $\sigma$--algebra so the integration theory that will follow, when considering (H2), will be valid for the open sets of $\tau_\Delta$. Nevertheless, while studying specific problems --like differential equations-- we will deal, in general, with intervals or other elements of the $\tau_u$--Borel $\sigma$--algebra without worrying about the specifics of the $\tau_\Delta$ topology, which, as said before, makes (H2) unneeded.
	\end{rem}

\end{exa}
\begin{dfn}\label{cont}
Given displacement spaces $(X,\Delta_1)$ and $(Y,\Delta_2)$, a function $f:X\to Y$ is said to be \emph{$\Delta_1^2$--continuous} if $f:(X,\tau_{\Delta_1})\to(Y,\tau_{\Delta_2})$ is continuous. 

We say that a map $f:X\to\bR^n$ is $\Delta_1$--continuous if $f:(X,\Delta_1)\to(\bR^n,\tau_u)$ is continuous.
\end{dfn}

As usual, continuity can be characterized using open balls, as it is shown in the following result.
\begin{lem}\label{lcont}
Let $(X,\Delta_1)$ and $(Y,\Delta_2)$ be displacement spaces. A map $f:X\to Y$ is $\Delta_1^2$--continuous if and only if 
\begin{equation}\label{epsdel}
\forall x\in X,\ \forall \varepsilon\in\bR^+\ \exists \delta\in\bR^+\text{ such that } f(y)\in B_{\Delta_2}(f(x),\varepsilon)\sfa y\in B_{\Delta_1}(x,\delta).
\end{equation}
\end{lem}
\begin{proof}
First, assume that $f$ is $\Delta_1^2$--continuous and fix $x\in\mathbb R$ and $\varepsilon\in\bR^+$. Since $U=B_{\Delta_2}(f(x),\varepsilon)\in\tau_{\Delta_2}$, we have that $f^{-1}(U)\in\tau_{\Delta_1}$. Moreover, $x\in f^{-1}(U)$ and so, there exists $\delta\in\bR^+$ such that $B_{\Delta_1}(x,\delta)\subset f^{-1}(U)$. Hence, $f(B_{\Delta_1}(x,\delta))\subset f(f^{-1}(U))\subset U$, that is, there exists $\delta\in\bR^+$ such that 
\begin{displaymath}|\Delta_1(x,y)|<\delta\implies |\Delta_2(f(x),f(y))|<\varepsilon.\end{displaymath}
Conversely, let $U\in\tau_{\Delta_2}$, $y\in f^{-1}(U)$ and $x=f(y)$. Since $U\in\tau_{\Delta_2}$, there exists $\e_x\in\bR^+$ such that $B_{\Delta_2}(x,\e_x)\subset U$. Now, condition~\eqref{epsdel} guarantees the existence of $\delta_y\in\bR^+$ such that
\[f(z)\in B_{\Delta_2}(x,\varepsilon_x),\sfa z\in B_{\Delta_1}(y,\delta_y).\]
Note that $B_{\Delta_1}(y,\delta_y)\subset f^{-1}(U)$ as for any $z\in B_{\Delta_1}(y,\delta_y)$ we have that $f(z)\in U$. Since $y\in f^{-1}(U)$ was arbitrary, $f^{-1}(U)$ is open and so $f$ is $\Delta_1^2$--continuous.
\end{proof}

\section{Displacement measure theory on the real line}\label{measure}

In this section we aim to define a measure over a non--degenerate interval $[a,b]\subset\bR$. To do so, we will use ``local'' measures $\mu_z$, for $z\in [a,b]$, to construct a measure $\mu$ which does not depend on a specific point $z$. In order to achieve that, we will consider $([a,b],\le,\Delta)$ satisfying hypotheses (H1)--(H3), and two extra conditions:
\begin{enumerate}
	\item[\textup{(H4)}] There exists $\c:[a,b]^2\to[1,+\infty)$ such that
	\begin{enumerate}
		\item [\textup{(i)}] For all $x,y,z,\overline z\in [a,b]$, we have
		\[|\Delta(z,x)-\Delta(z,y)|\le \gamma(z,\overline z) |\Delta(\overline z,x)-\Delta(\overline z,y)|.\]
		\item [\textup{(ii)}] For all $z\in [a,b]$,
		\[\lim_{\overline z\to z}\gamma(z,\overline z)=\lim_{\overline z\to z}\gamma(\overline z,z)=1.\]
		\item [\textup{(iii)}] For all $z\in [a,b]$, the maps $\gamma(z,\cdot),\gamma(\cdot,z):[a,b]\to[1,+\infty)$ are bounded.
	\end{enumerate}
	\item[\textup{(H5)}] For every $x\in [a,b],$ $\Delta_x(\cdot)$ is left--continuous (with the usual topology of $\bR$) at $x$.
\end{enumerate}

\begin{rem}
	Note that under hypothesis (H3), it is enough to check that there exists $\c:[a,b]^2\to[1,+\infty)$ such that, for all $x,y\in [a,b]$, $x<y$, we have
	\[\Delta(z,y)-\Delta(z,x)\le \gamma(z,\overline z) (\Delta(\overline z,y)-\Delta(\overline z,x)),\]
	to confirm that (H4) holds.
\end{rem}

	First of all, note that the set of maps $\Delta:[a,b]^2\to\mathbb R$ that satisfy hypotheses \textup{(H1)--(H5)} is not empty, as any Stieltjes displacement satisfies all of them. Moreover, there exist non--Stieltjes displacements that also satisfy all of the hypotheses. To show that this is the case, we will need the following result.

\begin{pro}\label{conddeltah23}
	Let $\Delta:[a,b] ^2 \to \mathbb R$ be a given map and let us denote by $D_2 \Delta$ its partial derivative with respect to its second variable. 
	If $D_2 \Delta$ exists and is continuous on $[a,b]^2$, and there exists $r>0$ such that
	\[D_2 \Delta (x,y) \ge r \quad \mbox{for all $(x,y) \in [a,b]^2$,}\]
	then $\Delta$ satisfies \textup{(H3)--(H5)}.
\end{pro}

\begin{proof}
	The assumptions imply that for each $x \in [a,b]$, the mapping $\Delta(x, \cdot)$ is increasing and continuous, which is more than (H3) and (H5). Now fix $z, \bar z \in [a,b]$. For $a \le x < y \le b$, the generalized mean value theorem guarantees the existence of $\xi \in (x,y)$ such that
	\[\dfrac{\Delta(z,y)-\Delta(z,x)}{\Delta(\bar z,y)-\Delta(\bar z,x)}=\dfrac{D_2 \Delta(z, \xi)}{D_2 \Delta(\bar z, \xi)},\]
	which implies (H4, i) for
	\begin{equation}
	\label{gammaDdelta}
	\gamma(z,\bar z)=\max\left\{ 1, \max_{a \le \xi \le b}\dfrac{D_2 \Delta(z, \xi)}{D_2 \Delta(\bar z, \xi)}\right\}.
	\end{equation}
	The function $\gamma$ is well--defined and bounded because the three variable mapping
	\[(z,\bar z,\xi) \in [a,b]^3 \longmapsto \dfrac{D_2 \Delta(z, \xi)}{D_2 \Delta(\bar z, \xi)}\]
	is continuous on a compact domain. In particular, (H4, iii) holds.

	Finally, (H4, ii) is a consequence of the fact that $D_2 \Delta$ is continuous on $[a,b]^2$, and, therefore, uniformly continuous on $[a,b]^2$. Indeed, let $z \in [a,b]$ be fixed; for $\varepsilon>0$ we can find $\delta>0$ such that for each $\bar z \in [a,b]$, $|z-\bar z| < \delta$, we have that
	\[|D_2 \Delta (z,\xi)-D_2 \Delta(\bar z, \xi)| < r \, \varepsilon \quad \mbox{for every $\xi \in [a,b]$.}\]
	Therefore, if $|z-\bar z| < \delta$, we have that
	\[\left| \dfrac{D_2 \Delta(z, \xi)}{D_2 \Delta(\bar z, \xi)}-1 \right| = \dfrac{|D_2 \Delta (z,\xi)-D_2 \Delta(\bar z, \xi)|}{D_2 \Delta(\bar z, \xi)}< \varepsilon \quad \mbox{for every $\xi \in [a,b]$.}\]
	We have just proven that
	\[\lim_{\bar z \to z}\dfrac{D_2 \Delta(z, \xi)}{D_2 \Delta(\bar z, \xi)}=1 \quad \mbox{uniformly in $\xi \in [a,b]$.}\]
	Now, for each $\bar z \in [a,b]$, there exists $\xi_{\bar z}$ such that
	\[\gamma(z,\bar z)=\max\left\{ 1, \dfrac{D_2 \Delta(z, \xi_{\bar z})}{D_2 \Delta(\bar z, \xi_{\bar z})} \right\},\]
	hence $\gamma(z,\bar z) \to 1$ as $\bar z \to z$. Similarly, $\gamma(z,\bar z) \to 1$ as $z \to \bar z$.
\end{proof}

\begin{exa}\label{ExDeltanoSti}
Consider the non--Stieltjes displacement in Example~\ref{dispexp}, namely $\Delta:[0,1]\times[0,1]\to\mathbb R$ where
\[\Delta(x,y)=e^{y^2-x^2}-e^{x-y},\quad x,y\in[0,1].\]
It clearly has continuous partial derivatives, and 
\begin{equation}\label{nonstider}
D_2 \Delta(x,y)=2y e^{y^2-x^2}+ e^{x-y} \ge e^{x-y} \ge e^{-1} \quad \mbox{on $[0,1]^2$.}
\end{equation}
Hence, $\Delta$ satisfies \textup{(H1)--(H5)} for $\gamma$ defined as in (\ref{gammaDdelta}).
\end{exa}

Although hypothesis (H5) might seem harmless, when combined with (H4), we obtain left--continuity everywhere.

\begin{pro}\label{plc}
	Consider $([a,b],\Delta)$ satisfying hypotheses \textup{(H4)} and \textup{(H5)}. Then, for each $x\in[a,b]$, the map $\Delta_x:[a,b]\to\bR$ is left--continuous everywhere (with the usual topology of $\bR$).
\end{pro}
\begin{proof}
	Let $\e>0$, $x,y\in[a,b]$ and $\gamma$ be the map on (H4). Let us show that $\Delta_x$ is left--continuous at $y$. Since, by (H5), $\Delta_y$ is left--continuous at $y$, there exists $\delta>0$ such that, for $0<y-s<\delta$, we have $|\Delta_y(y)-\Delta_y(s)|<\e/\gamma(x,y)$.
	Then, for $0<y-s<\delta$, hypothesis (H4) implies that
	\[|\Delta_x(y)-\Delta_x(s)|\le \gamma(x,y)|\Delta_y(y)-\Delta_y(s)|<\varepsilon.\qedhere\]
\end{proof}

 With the previous result in mind, we can define the ``local'' measures $\mu_z$ as the Lebesgue--Stieltjes measure associated with the non--decreasing left--continuous map $\Delta_z$. We shall denote by $\mathcal M_z$ the $\sigma$--algebra over which $\mu_z$ is defined. 


Let us denote by $\cB$ the Borel $\sigma$--algebra (for $\tau_u$) and by $\mathcal{M}:=\bigcap_{z\in[a,b]} \mathcal M_z$. Note that $\mathcal B\subset\mathcal M$ as $\mathcal B\subset\mathcal M_z$ for all $z\in[a,b]$. Moreover, $\mathcal M$ is a $\sigma$--algebra as it is an arbitrary intersection of $\sigma$--algebras. Hence, we can consider the restriction of $\mu_z$, $z\in[a,b]$, to $\mathcal M$. We will still denote it by $\mu_z$. A set $A\in\cM$ is said to be \emph{$\Delta$--measurable}.

 Recall that a function $f:([a,b],\cM)\to(\bR,\cB)$ is measurable if and only if $f^{-1}(U)\in \mathcal M$ for all $U\in\cB$. We will say in that case that $f$ is \emph{$\Delta$--measurable}. This notation will be consistent with the $\Delta$--measure that we will introduce later. Observe that $f$ is $\Delta$--measurable if and only $f:([a,b],\cM_z)\to(\bR,\cB)$ is measurable for all $z\in[a,b]$. 

Hypothesis (H4) allows us to understand the relationship between the different possible measures on $\cM$ depending on $z\in [a,b]$. In particular, given $z,\overline z\in [a,b]$, and an interval $I\subset[a,b]$, it is clear that 
$\mu_z(I)\leq\gamma(z,\overline z) \mu_{\overline z}(I)$
and, as a consequence of the definition of the Lebesgue--Stieltjes measures,
\[\mu_z(A)\leq \gamma(z,\overline z)\mu_{\overline z}(A),\quad\mbox{for all }A\in \mathcal M.\]
Thus, we have that $\mu_{\overline z}\ll \mu_z\ll\mu_{\overline z}$ for all $z,\overline{z}\in X$. Hence, if a property holds $\mu_z$--everywhere, it holds $\mu_{x}$--everywhere for all $x\in[a,b]$. Again, in order to simplify the notation, we will say that such property holds \emph{$\Delta$--everywhere}. Analogously, this expression will be consistent with the $\Delta$--measure presented later in this paper.

Then, given $z,\overline z\in[a,b]$ we can apply the Radon--Nikod\'ym Theorem \cite{Ben} to these measures, so there exist two $\Delta$--measurable functions $h_{\ol z,z}, h_{z,\ol z}:[a,b]\rightarrow[0,\infty)$ such that
\begin{equation}\label{hzzdef}
\mu_{\overline z}(A)=\int_A h_{\ol z,z}\dif\mu_{ z},\quad \mu_z(A)=\int_A h_{z,\ol z}\dif\mu_{\overline z},\quad \mbox{for all }A\in \mathcal M.
\end{equation}
From these expressions it is clear that $h_{z,z}=1$ and $h_{\widetilde z,z}(t)=h_{\widetilde z,\ol z}(t)h_{\ol z, z}(t)$ for $\Delta$--almost all (or simply $\Delta$--a.a.) $t\in[a,b]$ and $z,\ol z,\widetilde z\in[a,b]$. Hence, it follows that $h_{z,\ol z}(t)=1/h_{\ol z,z}(t)$ for $\Delta$--a.a. $t\in[a,b]$. Also note that for $z,\ol z\in[a,b]$, $h_{z,\ol z}\ne0$ $\Delta$--everywhere.

Further properties are shown in the next results.
\begin{pro}
	Given $z,\ol z\in [a,b]$, we have that
	\begin{equation}\label{hbound1}
	\frac{1}{\gamma(\ol z,z)}\le h_{z,\ol z}(t)\le\gamma(z,\ol z),\quad \Delta\mbox{--a.a. }t\in[a,b].
	\end{equation}
\end{pro}
\begin{proof}
	First, assume that $h_{z,\ol z}(t)\le\gamma(z,\ol z),$ does not hold for $\Delta\mbox{--a.a. }t\in[a,b]$. Then, there would exist $A\in\mathcal M$ such that $\mu_{\ol z}(A)>0$ and 
	\[h_{z,\ol z}(t)>\gamma(z,\ol z),\quad \mbox{for all }t\in A.\]
	Hence,
	\[\mu_z(A)=\int_A h_{z,\ol z}(s)\dif\mu_{\ol z}(s)> \int_A \gamma(z,\ol z)\dif\mu_z(s)=\gamma(z,\ol z)\mu_{\ol z}(A),\]
	which is a contradiction. Therefore,
	\begin{equation}\label{ineqhbound}
	h_{z,\ol z}(t)\le\gamma(z,\ol z),\quad \Delta\mbox{--a.a. }t\in[a,b].
	\end{equation} 
	For the other inequality, take $h_{\ol z,z}$ as in~\eqref{hzzdef}. Using~\eqref{ineqhbound}, we have that
	\[\gamma(\ol z,z)\ge h_{\ol z,z}(t)=\frac{1}{h_{z,\ol z}(t)},\quad \Delta\mbox{--a.a. }t\in[a,b],\]
	from which the result follows.
\end{proof}

Note that this result yields that, for $z,\ol z\in[a,b]$ fixed, we have that
\begin{equation*}
\frac{1}{\gamma(\ol z,z)}\le h_{z,\ol z}(t)\le\gamma(z,\ol z),\quad \mbox{for all }t\in[a,b]\setminus A_z,
\end{equation*}
with $\mu_z(A_z)=0$. Let us define $\tilde h_{z,\ol z}:[a,b]\to [0,+\infty)$ as
\[
\tilde h_{z,\ol z}(t)=\left\{
\begin{array}{lcl}
h_{z,\ol z}(t),&\mbox{if}& t\in[a,b]\setminus A_z,\\
1,&\mbox{if}& t\in A_z.
\end{array}
\right.
\]
Then, it follows from~\eqref{hbound1} that
\begin{equation}\label{hbound}
\frac{1}{\gamma(\ol z,z)}\le \tilde h_{z,\ol z}(t)\le\gamma(z,\ol z),\quad \mbox{for all }t\in[a,b].
\end{equation}
Moreover, $\tilde h_{z,\ol z}=h_{z,\ol z}$ $\Delta$--a.e. in $[a,b]$, and in particular, $\mu_z$--a.e. in $[a,b]$. Hence, we have that
\[\mu_z(A)=\int_{A}\tilde h_{z,\ol z}\dif\mu_{\ol z},\quad\mbox{for all }A\in\mathcal M.\]
Thus, we can assume without loss of generality that the functions in~\eqref{hzzdef} satisfy~\eqref{hbound}.
Given this consideration, we can obtain the following result.
\begin{pro}
	For all $t\in[a,b]$, we have that
	\begin{equation}\label{hlim}
	\lim_{z\to \ol z}h_{z,\ol z}(t)=1.
	\end{equation}
\end{pro}
\begin{proof}
	Fix $\ol z, t\in[a,b]$. Then from~\eqref{hbound} we obtain that
	\begin{equation}\label{hboundz}
	\frac{1}{\gamma(\ol z,z)}\le h_{z,\ol z}(t)\le\gamma(z,\ol z),\quad \mbox{for all }z\in[a,b].
	\end{equation}
	Hence it is enough to consider the limit when $z\to\ol z$ in the previous inequalities, together with hypothesis \textup{(H3, ii)}, to obtain the result.
\end{proof}

\begin{rem}\label{hmbound}
	Note that, given $\overline z\in[a,b]$, we also have that there exist $m_{\ol z},M_{\ol z}>0$ such that 
	\[m_{\ol z}\le h_{t,\overline z}(t)\le M_{\ol z},\quad \mbox{for all }t\in[a,b].\]
	Indeed, fix $\ol z, t\in[a,b]$. Then,~\eqref{hboundz} with $z=t$ yields
	\[\frac{1}{\gamma(\ol z,t)}\le h_{t,\ol z}(t)\le\gamma(t,\ol z).\]
	Now the result follows from \textup{(H3, iii)}.
\end{rem}

We will now focus on the definition of the $\Delta$--measure which is based on the integrals defined by the measures $\mu_z$, $z\in[a,b]$. We first will show that a bigger family of maps is well--defined.

\begin{pro}
	Let $\a:([a,b],\cB)\to([a,b],\cB)$ be a measurable map and $z\in[a,b]$. Then the map $\mu_\a:\mathcal M\to [0,+\infty]$ given by
	\[\mu_\a(A)=\int_A h_{\a(t),z}(t)\dif\mu_z(t),\quad A\in\mathcal M,\]
	 is well--defined, that is, $h(\cdot,\a(\cdot)):=h_{\a(\cdot),z}(\cdot)$ is $\Delta$--measurable.
\end{pro}
\begin{proof}
	In order to show that $h(\cdot,\a(\cdot))$ is $\mu_z$--measurable, let us define the map $h_z:[a,b]^2\to[0,+\infty)$ given by $h_z(t,x)=h_{x,z}(t)$. We will first show that $h_z$ is a $\Delta_z$--Carath\'eodory in the sense of \cite[Definition 7.1]{FP2016} adapted to our notation, that is:
	\begin{enumerate}
		\item[\textup{(i)}] for every $x\in[a,b]$, $h_z(\cdot,x)$ is $\Delta$--measurable;
		\item[\textup{(ii)}] for $\Delta$--a.a. $t\in[a,b]$, $h_z(t,\cdot)$ is continuous on $[a,b]$;
		\item[\textup{(iii)}] for every $r>0$, there exists $f_r\in \mathcal L^1_{\Delta_z}([a,b))$ such that
		\[|h_z(t,x)|\le f_r(t)\quad \mbox{for $\Delta$--a.a. }t\in[a,b),\mbox{ and for all }x\in X, |x|\le r.\]
	\end{enumerate}
	Note that condition (i) is trivial as, by definition, $h_z(\cdot,x)=h_{x,z}(\cdot)$ is $\Delta$--measurable. As for condition (ii), let $x\in[a,b]$. It follows directly from~\eqref{hlim} that
	\[\lim_{\ol z\to x}h_{\ol z,z}(t)=\lim_{\ol z\to x}h_{\ol z,x}(t)h_{x,z}(t)=h_{x,z}(t),\quad \mbox{for }\Delta\mbox{--a.a. }t\in[a,b],\]
	that is, $h_z(t,\cdot)$ is continuous on $[a,b]$ for $\Delta$--a.a. $t\in[a,b].$
	Finally, (H4, iii) guarantees the existence of $M>0$ such that $|\gamma(x,z)|<M$ for all $x\in [a,b]$. Hence, it follows from~\eqref{hbound} that
	\[|h_z(t,x)|\le M,\quad \mbox{for $\Delta$--a.a. }t\in[a,b],\mbox{ for all }x\in X,\ |x|\leq r.\]
	Thus, (iii) holds, i.e., the map $h_z$ is $\Delta_z$--Carath\'eodory. Now, as it is shown in \cite[Lemma 7.2]{FP2016}, the composition of a $\Delta_z$--Carath\'eodory function with a Borel measurable function is $\Delta$--measurable, and so the result follows.
\end{proof}

\begin{dfn}\label{dmeasure}
	Let $\a:([a,b],\cB)\to([a,b],\cB)$ be a measurable map and $z\in[a,b]$. Consider the map $\mu_\a:\mathcal M\to [0,+\infty)$ given by
	\[\mu_\a(A)=\int_A h_{\a(t),z}(t)\dif\mu_z(t),\quad A\in\mathcal M.\]
	The map $\mu_{\a}$ is a measure, \cite[Theorem 1.29]{rudin}, and it will received the name of \emph{$\Delta_{\a}$--measure}. In particular, when $\a$ is the identity map, it will be called the \emph{$\Delta$--measure}, and it will be denoted by $\mu\equiv \mu_{\Id}$.
\end{dfn}
The following result shows that $\mu_{\a}$ does not depend on a specific point of $[a,b]$, as we intended. 

\begin{pro}
	Let $\a:([a,b],\cB)\to([a,b],\cB)$ be a measurable map. Then the map $\mu_\a:\mathcal M\to [0,+\infty]$ in Definition~\ref{dmeasure}
	is independent of the choice of $z\in[a,b]$.
\end{pro}
\begin{proof}	
	Let $z,\ol z\in[a,b]$, $z\not=\ol z$. Then, we have that $h_{\a(s),\ol z}(s)=h_{\a(s),z}(s)h_{z,\ol z}(s)$ for $\Delta$--a.a. $s\in[a,b]$ and so
	\[\int_A h_{\a(s),z}(s)\dif \mu_z(s)=\int_A h_{\a(s),z}(s)h_{z,\ol z}(s)\dif \mu_{\ol z}(s)=\int_Ah_{\a(s),\ol z}(s) \dif\mu_{\ol z}(s).\qedhere\]
\end{proof}

Observe that the notation we have used so far is consistent with the definition of the $\Delta$--measure $\mu$. Indeed, for example, $f$ is \emph{$\Delta$--measurable} if and only if it is \emph{$\mu$--measurable}; as $\mu$ and $\mu_z$, $z\in[a,b]$ are both defined, after due restriction of $\mu_z$ to $\cM$, over the same $\sigma$--algebra. Also, by definition, we have that $\mu\ll \mu_z$. The converse is also true thanks to (H4, iii). Indeed, for $z\in[a,b]$, there exists $K>0$ such that $|\gamma(z,x)|<K$ for all $x\in[a,b]$. Hence,~\eqref{hbound} implies that
\[\mu(A)=\int_A h_{\a(s),z}(s)\dif\mu_z(s)\ge \int_A\frac{1}{\gamma(z,\a(s))}\dif \mu_z(s)\ge\int_A\frac{1}{K}\dif \mu_z(s)=\frac{\mu_z(A)}{K}\ge 0.\]
Thus, if $\mu(A)=0$ then $\mu_z(A)=0$, i.e., $\mu_z\ll\mu$ for all $z\in[a,b]$. Therefore, a property holds \emph{$\Delta$--everywhere} if and only if it holds \emph{$\mu$--everywhere}.

As a final comment, note that $\mu:\mathcal M\to[0,+\infty]$ is a Borel measure that assigns finite measure to bounded sets. As it can be seen in \cite[Chapter 1, Section 3, Subsection 2]{athreya}, this means that it can be thought of as a Lebesgue--Stieltjes measure, $\mu_g$, given by 
\begin{equation}\label{gmeas}
\mu_g([c,d))=g(d)-g(c),\quad c,d\in[a,b],\ c<d,
\end{equation}
for the nondecreasing and left--continuous function $g:[a,b]\to\bR$ defined as \begin{equation}\label{gfunc}
g(a)=0,\quad g(t)=\mu([a,t)).
\end{equation}

\begin{dfn}
	Let $X\subset[a,b]$, $X\subset \mathcal M$. 	
	We define the \emph{integral of a $\mu_\a$--measurable function $f$ over $X$ with respect to the path of $\Delta$--measures $\a$ }as
	\[\int_X f\dif\mu_\a:=\int_X f(t)h_{\a(t),z}(t)\dif\mu_z(t),\]
	provided the integral exists. This definition does not depend on the $z$ chosen.
	As usual, we define the set of \emph{$\mu_\a$--integrable functions} on $X$ as
	\[\mathcal L^1_{\mu_\a}(X):=\left\{f:X\to\bR: f\mbox{ is }\mu_\a\mbox{--measurable},\ \int_X |f|\dif \mu_\a<+\infty \right\}.\]
\end{dfn}

Now, if we consider the restrictions of $\mu_z$, $z\in[a,b]$, to $\mathcal M$, we can define the set of \emph{$\Delta$--integrable functions over $X\in\mathcal M$} as
\[\mathcal L_\Delta^1(X):=\left\{f:X\to\bR: f\mbox{ is }\Delta\mbox{--measurable},\, \int_X |f|\dif\mu_z<+\infty,\mbox{ for all } z\in[a,b]\right\}.\]
We now study the relationship between $\mathcal L_\Delta^1(X)$ and $\mathcal L_\mu^1(X)$. First of all, recall that, in this framework, $\mu$ and $\mu_z$ are defined over the same $\sigma$--algebra $\mathcal M$, so the concepts of $\mu$--measurable and $\Delta$--measurable are equivalent. Let $f\in\mathcal L_\Delta^1(X)$ and $z\in[a,b]$. Hypothesis (H4, iii) implies that there exist $M>0$ such that $|\gamma(x,z)|<M$ for all $x\in[a,b]$. Thus, using~\eqref{hbound} we have
\[
\int_{X}|f|\dif \mu=\int|f(s)h_{s,z}(s)|\dif\mu_z(s)\le\int|f(s)\gamma(s,z)|\dif\mu_z(s)\le
M\int|f(s)|\dif\mu_z(s)<+\infty,\]
that is, $f\in\mathcal L_\mu^1(X)$. Conversely, let $f\in\mathcal L_\mu^1(X)$ and $z\in[a,b]$. Again, hypothesis (H4, iii) implies that there exists $K>0$ such that $|\gamma(z,x)|<K$ for all $x\in[a,b]$. Hence
\[
\int_{X}|f|\dif \mu=\int|f(s)h_{s,z}(s)|\dif\mu_z(s)\ge\int\left|\frac{f(s)}{\gamma(z,s)}\right|\dif\mu_z(s)\ge
\frac{1}{K}\int|f(s)|\dif\mu_z(s),\]
so $\int_X|f|\dif\mu_z<+\infty$. Since $z\in[a,b]$ was arbitrary, we have that $f\in \mathcal L_\Delta^1(X)$. That is, $\mathcal L_\mu^1(X)=\mathcal L_\Delta^1(X)$.

Finally, we study the behavior of $\mu$ over some interesting sets related to the map $\Delta$. These sets will be fundamental in the definition of the $\Delta$--derivative. Let us define the sets $C_\Delta$ and $D_\Delta$ as
\begin{align}\label{cdelta}
C_\Delta:= &\{x\in [a,b]:\Delta(x,\cdot)=0\mbox{ in }(x-\e,x+\e)\mbox{ for some }\e\in\bR^+\},\\\label{ddelta}
D_\Delta:= &\{x\in[a,b]: \Delta(x,x^+)\not=0\}.\notag
\end{align}

Note that $C_\Delta$ is, by definition, an open set in the usual topology of $[a,b]$. Therefore, it can be rewritten uniquely as the disjoint countable union of open intervals, say $C_\Delta=\bigcup_{n\in\bN} (a_n,b_n)$. We define $N_\Delta$ as
\begin{equation}\label{ndelta}
N_\Delta:=\{a_n,b_n:n\in\bN\}\setminus D_\Delta.
\end{equation}

\begin{pro}\label{cdrel}
	Let $g$ be as in~\eqref{gfunc} and let $C_g$ and $D_g$ be as in \textup{\cite{PoRo}}, that is,
	\begin{align*}C_g:= & \{x\in[a,b]: g\mbox{ is constant on }(x-\e,x+\e)\mbox{ for some }\e\in\bR^+\},\\ D_g:= & \{x\in[a,b]:g(x^+)-g(x)>0\}.\end{align*}
	Then $C_\Delta=C_g$ and $D_\Delta=D_g$ for the Stieltjes displacement given by $g$.
\end{pro}
\begin{proof}
	For the equality $D_\Delta=D_g$, it is enough to note that for any $t\in[a,b]$ we have that
	\[g(t^+)-g(t)=\lim_{s\to t^+}\int_{[t,s)}h_{r,t}(r)\dif\mu_t(r)=\int_{\{t\}}h_{r,t}(r)\dif\mu_t(r)=h_{t,t}(t)(\Delta_t(t^+)-\Delta_t(t))=\Delta_t(t^+).\]

	Now, in order to see that $C_\Delta=C_g$, let $t\in C_\Delta$. Then $\Delta_t(\cdot)=0$ on $(t-\e,t+\e)$ for some $\e\in\bR^+$. Let $r,s\in(t-\e,t+\e)$, $r<s$. Then, by Remark~\ref{hmbound}, we have that
	\[0\le g(s)-g(r)=\mu([r,s))=\int_{[r,s)}h_{x,t}(x)\dif\mu_t(x)\le M_t\mu_t([r,s))=0,\]
	since $[r,s)\subset (t-\e,t+\e)$. Thus $g$ is constant on $(t-\e,t+\e)$. Conversely, if $t\in C_g$, then $g$ is constant on $(t-\e,t+\e)$ for some $\e\in\bR^+$. Let $r,s\in(t-\e,t+\e)$, $r<s$. Then, Remark~\ref{hmbound} implies that
	\[0\le m_t(\Delta_t(s)-\Delta_t(r))=m_t\mu_t([r,s))\le \int_{[r,s)}h_{x,t}(x)\dif\mu_t(x)=g(s)-g(r)=0.\]
	That is, $\Delta_t$ is constant on $(t-\e,t+\e)$, and since $\Delta_t(t)=0$, it follows that $t\in D_\Delta$.
\end{proof}

The first consequence of Proposition~\ref{cdrel} is that $D_\Delta$ is at most countable since it is the set of discontinuities of a monotone function. Further properties can be obtain from Propositions 2.5 and 2.6 in \cite{PoRo}.
\begin{cor}\label{corcdrel}
	Let $C_\Delta$ and $N_\Delta$ be as in~\eqref{cdelta} and~\eqref{ndelta}, respectively. Then $\mu(C_\Delta)=\mu(N_\Delta)=0$.
\end{cor}

\begin{rem}\label{nullm}
	As a consequence of Corollary~\ref{corcdrel}, a property holds $\mu$--a.e. in $E$ if it holds on $E\setminus O_\Delta$ with \[O_\Delta:=C_\Delta\cup N_\Delta.\] Moreover, note that, if $x\not\in O_\Delta$, then $\Delta_z(y)\not=\Delta_z(x)$, for all $y,z\in[a,b]$, $x\not= y$.
\end{rem}

\section{Displacement derivatives}\label{derivatives}
We now introduce the concept of displacement derivative of a function defined over a compact interval endowed with a displacement structure in the real line equipped with the usual topology. We chose this setting because some nice properties, such as the linearity of the derivative, are quite helpful in order to study the relationship between the displacement derivative and its integral.
\begin{dfn}\label{deriv} 
	Let $([a,b],\le,\Delta)$ satisfy \textup{(H1)--(H5)}.
	The \emph{derivative with respect to the displacement $\Delta$} (or \emph{$\Delta$--derivative}) of a function $f:[a,b]\to \bR$ at a point $x\in [a,b]\setminus O_{\Delta}$ is defined as follows, provided that the corresponding limits exist:
\begin{align*} f^{\Delta}(x) =\begin{dcases}\lim_{y \to x}\frac{f(y)-f(x)}{\Delta(x,y)}, & x\not\in D_{\Delta},\\\lim_{y \to x^+}\frac{f(y)-f(x)}{\Delta(x,y)}, & x\in D_{\Delta}.\end{dcases}
\end{align*}
\end{dfn}

Observe that this definition does not require $\Delta$ to be symmetric. Furthermore, this definition is a more general setting than $g$--derivatives (and therefore time--scales, as pointed out in \cite{PoRo}).

Finally, one might think that the natural choice for the definition of the derivative would be by taking the limit in the $\tau_\Delta$ topology. However, if $x\not\in D_\Delta$, $x$ is a continuity point of $\Delta_x$, and it is easy to see that such limit can be translated into a limit in the usual topology, which is far more convenient for the theory that follows. It is at this point that the importance of (H1) arises as commented in Remark~\ref{h1r}. Without this hypothesis we would not be able to assure that the balls of center $x$ and any radii are nonempty, so considering the $\tau_\Delta$ limit might not be well--defined.


\subsection{Fundamental Theorem of Calculus}
In this section we will make explicit the relationship between the $\Delta$--derivative of $f$ and its integral with respect to the $\Delta$--measure. In particular, our first goal now is to show that, for $f\in\mathcal L_{\mu}^1([a,b])$ and $F(x)=\int_{[a,x)} f(s)\dif\mu$, the equality $F^{\Delta}(x)=f(x)$ holds for $\mu$--a.a. $x\in[a,b]$. In the particular setting of Stieltjes derivatives, this result has been proven in different ways and can be found in \cite{Daniell1918} and, more extensively, in \cite{Garg1992}.

In order to do so, we will follow an approach similar to that of \cite{PoRo}, starting by guaranteeing the differentiability of monotone functions. For that matter, we will use the following two results that are direct consequences of Lemmas 4.2 and 4.3 in \cite{PoRo} adapted to our framework.
\begin{pro}\label{lem1}
	Let $([a,b],\le,\Delta)$ satisfy \textup{(H1)--(H5)}, $\a>0$, $z\in[a,b]$, $f:[a,b]\to\bR$, $P=\{x_0,x_1,\dots, x_n\}$ be a partition of $[a,b]$ and $S$ be a nonempty subset of $\{1,2,\dots,n\}$. If $f(a)\le f(b)$ and
	\[\frac{f(x_k)-f(x_{k-1})}{\Delta_z(x_k)-\Delta_z(x_{k-1})}<-\a\quad\mbox{for each }k\in S,\]
	then
	\[\sum_{k=1}^n |f(x_k)-f(x_{k-1})|>|f(b)-f(a)|+\a L\]
	where $L=\sum_{k\in S}(\Delta_z(x_k)-\Delta_z(x_{k-1}))$. The same result is true if $f(a)\ge f(b)$ and 
	\[\frac{f(x_k)-f(x_{k-1})}{\Delta_z(x_k)-\Delta_z(x_{k-1})}>\a\quad\mbox{for each }k\in S.\]
\end{pro}

\begin{pro}\label{lem2}
 Let $([a,b],\le,\Delta)$ satisfy \textup{(H1)--(H5)} and $H\subset (a,b)$ be such that for a given $z\in[a,b]$, there is $\e_z\in\bR^+$ such that $\mu_z^*(H)=\e_z$. Then
 \begin{enumerate}
 	\item If $\mathcal I$ is any collection of open subintervals of $[a,b]$ that covers $H$, then there exists a finite disjoint collection $\{ I_1,I_2,\dots,I_N\}$ of $\mathcal I$ such that
 	\[\sum_{k=1}^N\mu_z(I_k)>\frac{\e_z}{3}.\]
 	\item If $P$ is a finite subset of $[a,b]\setminus D_\Delta$ and $\mathcal I$ is any collection of open subintervals of $[a,b]$ that covers $H\setminus P$, then there exists a finite disjoint collection $\{ I_1,I_2,\dots,I_N\}$ of $\mathcal I$ such that
 	\[\sum_{k=1}^N\mu_z(I_k)>\frac{\e_z}{4}.\]
 \end{enumerate}
\end{pro}
We can now prove the $\Delta$--differentiability of monotone functions. To do so, we will follow the ideas of \cite{Botsko2003}.
\begin{pro}\label{mondif}
	Let $([a,b],\le,\Delta)$ satisfy \textup{(H1)--(H5)} and let $f:[a,b]\to\bR$ be a nondecreasing function. Then there exists $N\subset [a,b]$ such that $\mu(N)=0$ and
	\[f^\Delta(x)\mbox{ exists for all }x\in[a,b)\setminus N.\]
\end{pro}
\begin{proof}
	First of all, note that, since $f$ is nondecreasing, f is regulated so if $x\in [a,b)\cap D_\Delta$, $f^\Delta(x)$ exists as
	\[f^\Delta(x)=\frac{f(x^+)-f(x)}{\Delta(x,x^+)}.\]
	Moreover, since either $a\in D_\Delta$ or $\mu(\{a\})=0$, it is enough to show that $f^\Delta(x)$ exists for all $x\in (a,b)\setminus(D_\Delta\cup O_\Delta)$ according to Remark~\ref{nullm}. 

	Let $x\in(a,b)\setminus(D_\Delta\cup O_\Delta)$. Since $\Delta_x(y)\not= \Delta_x(x)$ for any $y\not=x$, we can define the Dini upper and lower $\Delta$--derivatives as
	\[\overline{f^\Delta}(x):=\limsup_{y\to x}\frac{f(y)-f(x)}{\Delta(x,y)},\quad \underline{f^\Delta}(x):=\liminf_{y\to x}\frac{f(y)-f(x)}{\Delta(x,y)}.\]
	Furthermore, since $f$ is monotone, it has a countable number of discontinuity points, so 
	\[\mu(\{x\in(a,b)\setminus(D_\Delta\cup O_\Delta): f \mbox{ is discontinuous at }x\})=0.\]
	Thus, it is enough to show that the sets 
	\begin{align*}F:= &\{x\in(a,b)\setminus(D_\Delta\cup O_\Delta): f \mbox{ continuous at }x,\ \overline{f^\Delta}(x)>\underline{f^\Delta}(x) \},\\
	E: = & \{x\in(a,b)\setminus(D_\Delta\cup O_\Delta):\overline{f^\Delta}(x)=+\infty\},	
\end{align*}
	both have $\Delta$--measure zero (and therefore $\mu$--measure zero).

	We first show that $F$ is a null $\Delta$--measure set. Fix $z\in(a,b)\setminus(D_\Delta\cup O_\Delta)$ and define the Dini upper and lower $\Delta_z$--derivatives as
	\[\overline{f^\Delta_z}(x):=\limsup_{y\to x}\frac{f(y)-f(x)}{\Delta_z(y)-\Delta_z(x)},\quad \underline{f^\Delta_z}(x):=\liminf_{y\to x}\frac{f(y)-f(x)}{\Delta_z(y)-\Delta_z(x)}.\]
	Note that (H4) implies that
	\begin{align}
	\overline{f^\Delta}(x)&=\limsup_{y\to x}\frac{f(y)-f(x)}{\Delta_x(y)-\Delta_x(x)}=\limsup_{y\to x}\frac{f(y)-f(x)}{\Delta_x(y)-\Delta_x(x)}\frac{\Delta_z(y)-\Delta_z(x)}{\Delta_z(y)-\Delta_z(x)}\nonumber\\
	&=\overline{f^\Delta_z}(x)\limsup_{y\to x}\frac{\Delta_x(y)-\Delta_x(x)}{\Delta_z(y)-\Delta_z(x)}\leq \overline{f^\Delta_z}(x)\gamma(z,x), \nonumber
	\end{align}
	and, analogously, $\underline{f^\Delta}(x)\ge(\gamma(x,z))^{-1}\underline{f^\Delta_z}(x)$. Hence, $F$ is a subset of
	\[F_z:=\{(a,b)\setminus(D_\Delta\cup O_\Delta):f \mbox{ continuous at }x,\ \overline{f^\Delta_z}(x)\gamma(z,x)\gamma(x,z)>\underline{f^\Delta_z}(x)\}.\]
	Now, since $\gamma:[a,b]^2\to[1,+\infty)$, it is clear that $F_z\subset \bigcup_{n\in\bN} F_n$ with
	\[F_n:=\{(a,b)\setminus(D_\Delta\cup O_\Delta):f \mbox{ continuous at }x,\ \overline{f^\Delta_z}(x) n>\underline{f^\Delta_z}(x)\},\]
	so, it suffices to show that $\mu_z(F_n)=0$ for all $n\in\bN$. By contradiction, assume that there exists $n_0\in \bN$ such that $\mu_z(F_{n_0})>0$. In that case, we rewrite $F_{n_0}$ as the countable union of sets $F_{n_0,r,s}$ with $r,s\in\bQ$, $r>s>0$, and
	\[F_{n_0,r,s}:=\left\{x\in F_{n_0}:\overline{f^\Delta_z}(x) n_0>r>s>\underline{f^\Delta_z}(x) \right\}.\]
	Thus, there exist $r_0,s_0\in\bQ$, $r_0>s_0>0$, and $\e\in\bR^+$ such that $\mu_z(F_{n_0,r_0,s_0})=\e$. Now let $\a=\frac{r_0-s_0}{2n_0}$, $\b=\frac{r_0+s_0}{2n_0}$ and $h(x)=f(x)-\beta \Delta_z(x)$. Then $F_{n_0,r_0,s_0}=H$ with
	\[H:=\left\{x\in F_{n_0}:\overline{h^\Delta_z}(x)>\a,\ \underline{h^\Delta_z}(x)<-\a\right\}.\]

	Note that $h$ is of bounded variation as it is the difference of two nondecreasing functions. Therefore, the set \[V(h):=\left\{\sum_P |h(x_k)-h(x_{k-1})|: P\mbox{ is a partition of }[a,b],\ P\cap D_\Delta\subset \{a,b\}\right\}\] is bounded from above. Let $T:=\sup V(h)$. Since $\a,\e\in\bR^+$, there exists a partition $P=\{x_0,x_1,\dots,n-1\}$ of $[a,b]$ such that $x_k\not\in D_\Delta$ for any $k\in\{1,2,\dots,n-1\}$ and
	\[\sum_{k=1}^n |h(x_k)-h(x_{k-1})|> T-\frac{\a\e}{4}.\]
	Let $x\in H\setminus P$. Then $x\in (x_{k-1},x_k)$ for some $k\in\{1,2,\dots,n\}$. Since $\overline{h^\Delta_z}(x)>\a$, $\underline{h^\Delta_z}(x)<-\a$ and both $\Delta_z$ and $h$ are continuous at $x$, we can choose $a_x$, $b_x\in (x_{k-1},x_k)\setminus D_\Delta$ such that $a_x<x<b_x$ and
	\[\frac{h(b_x)-h(a_x)}{\Delta_z(b_x)-\Delta_z(a_x)}<-\a\quad\mbox{or}\quad >\a,\]
	depending on whether $h(x_{k-1})\geq h(x_k)$ or $h(x_{k-1})< h(x_k)$. Note that $\mu_z(a_x,b_x)=\Delta_z(b_x)-\Delta_z(a_x).$
	By doing this, we obtain a collection of open subintervals of $(a,b)$, $\mathcal I=\{(a_x,b_x): x\in H\setminus\{x_1,x_2,\dots,x_{n-1}\}\}$, that covers $H\setminus\{x_1,x_2,\dots,x_{n-1}\}$ and $\{x_1,x_2,\dots,x_{n-1}\}\cap D_\Delta=\emptyset$. Then, Proposition~\ref{lem2} ensures the existence of a finite disjoint subcollection $\{I_1,I_2,\dots,I_N\}$ of $\mathcal I$ such that
	\[\sum_{k=1}^N \mu_z(I_k)>\frac{\e}{4}.\]
	Now let $Q=\{y_0,y_1,\dots, y_q\}$ be the partition of $[a,b]$ determined by the points of $P$ and the endpoints of the intervals $I_1,I_2,\dots,I_N$. For each $[x_{k-1},x_k]$ containing at least one of the intervals in $\{I_1,I_2,\dots,I_N\}$, Proposition~\ref{lem1} yields that
	\[\sum_{[y_{i-1},y_i]\subset [x_{k-1},x_k]}|h(y_i)-h(y_{i-1})|>|h(x_k)-h(x_{k-1})|+\a L_k,\]
	where the summation is taken over the closed intervals determined by $Q$ contained in $[x_{k-1},x_k]$ and $L_k$ is the sum of the $\Delta_z$--measures of those intervals $I_1,I_2,\dots,I_N$ contained in $[x_{k-1},x_k]$. By taking the previous inequality and summing over $k$, we obtain
	\[\sum_{k=1}^q|h(y_k)-h(y_{k-1})|>\sum_{k=1}^n|h(y_k)-h(y_{k-1})|+\a\sum_{k=1}^N \mu_z(I_k)>T,\]
	which contradicts the definition of $T$.

	Hence, all that is left to do is to show that the set $E$ has $\Delta$--measure zero. If we fix $z\in[a,b]$, then for all $x\in(a,b)\setminus(D_\Delta\cup O_\Delta)$ we have the inequality $\overline{f^\Delta}(x)\le\overline{f^\Delta_z}(x)\gamma(z,x)$ and so $E\subset E_z$ with
	\[E_z:=\{x\in(a,b)\setminus(D_\Delta\cup O_\Delta): \overline{f^\Delta_z}(x)=+\infty\}.\]
	Thus, it is enough to show that $\mu_z(E_z)=0$. Suppose this is not the case. Then there is $\e\in \bR^+$ such that $\mu_z(E_z)=\e.$ Let $M\in \bR^+$ be such that $M>3(f(b)-f(a))/\e$.
	If $x\in E_z$ then $\overline{f^\Delta_z}(x)>M$ and there exist $a_x,b_x\in(a,b)\setminus D_\Delta$ such that $a_x<x<b_x$ and
	\[\frac{f(b_x)-f(a_x)}{\Delta_z(b_x)-\Delta_z(a_x)}>M.\]
	Therefore, $\{(a_x,b_x):x\in E_z\}$ covers $E_z$. Proposition~\ref{lem2} guarantees the existence of a finite disjoint subcollection $\{I_1,I_2,\dots,I_N\}$ such that
	\[\sum_{k=1}^N \mu_z(I_k)>\frac{\e}{3}.\]
	Let $I_k=(a_k,b_k)$ for each $k$. Then $\mu_z(I_k)=\Delta_z(b_k)-\Delta_z(a_k)$ as each $I_k\subset (a,b)\setminus(D_\Delta\cup O_\Delta)$. Now, since $f$ is nondecreasing we have
	\[f(b)-f(a)\ge \sum_{k=1}^N (f(b_k)-f(a_k))>M \sum_{k=1}^N(\Delta_z(b_k)-\Delta_z(a_k))>f(b)-f(a),\]
	which is a contradiction.
\end{proof}

Finally, a key result for the proof of the Fundamental Theorem of Calculus is Fubini's Theorem on almost everywhere differentiation of series for $\Delta$--derivatives. We now state such result but we omit its proof as it is essentially the one provided in \cite{Stro} but using Proposition~\ref{mondif} instead of the classical Lebesgue Differentiation Theorem.
\begin{pro}\label{fub}
	Let $([a,b],\le,\Delta)$ satisfy \textup{(H1)--(H5)} and let $(f_n)_{n\in\bN}$ be a sequence of real--valued nondecreasing functions on $[a,b]$. If the series
	\[f(x)=\sum_{n=1}^\infty f_n(x)\quad\mbox{converges for all } x\in [a,b],\]
	then
	\[f^\Delta(x)=\sum_{n=1}^\infty f_n^\Delta(x)\quad\mbox{for } \mu\mbox{--a.a. } x\in [a,b].\]
\end{pro}

We now have all the necessary tools to state and prove the first part of the Fundamental Theorem of Calculus for $\Delta$--derivatives.

\begin{thm}[Fundamental Theorem of Calculus]\label{FTC}
Let $f\in\mathcal L_{\mu}^1([a,b))$ and $F(x)=$ $\int_{[a,x)} f(s)\dif\mu$. 
Then $F^{\Delta}(x)=f(x)$ for $\mu$--a.a. $x\in[a,b]$.
\end{thm}
\begin{proof} 
	Without loss of generality we can assume that $f\ge 0$, as the general case can be reduced to the difference of two such functions. Since $f\ge 0$,
	the function $F$ is nondecreasing and therefore $\Delta$--differentiable. We consider several cases separately:

	\noindent
	{\it Case 1:} If $f=\chi_{(\a,\b)}$, where $(\a,\b) \subset (a,b)$, $\a, \b \not \in D_\Delta$, then it is obvious that $F^\Delta(x)=0=\chi_{(\a,\b)}(x)$ for $x\not\in(\a,\b).$ Now, if $x\in(\a,\b)$ let $H:[a,b]\to\bR$ be defined as 
	\[H(t)=\int_{[a,t)}h_{s,x}(s)\dif \mu_x(s),\]
	where the integral is to be understood as a Lebesgue--Stieltjes integral. Note that $H$ is well--defined. Then
	\begin{align*}
	F^\Delta(x)&=\lim_{y \to x^+}\frac{F(y)-F(x)}{\Delta(x,y)}=\lim_{y \to x^+}\frac{\int_{[a,y)}f\dif\mu-\int_{[a,x)}f\dif\mu}{\Delta(x,y)}\\
	&=\lim_{y \to x^+}\frac{\int_{[a,y)}h_{s,x}(s)f(s)\dif\mu_x(s)-\int_{[a,x)}h_{s,x}(s)f(s)\dif\mu_x(s)}{\Delta(x,y)}\\
	&=\lim_{y \to x^+}\frac{\int_{[x,y)}h_{s,x}(s)f(s)\dif\mu_x(s)}{\Delta(x,y)}=\lim_{y \to x^+}\frac{H(y)-H(x)}{\Delta_x(y)-\Delta_x(x)}\\
	&= H'_{\Delta_x}(x)=h_{x,x}(x)=1=\chi_{(\a,\b)}(x)
	\end{align*}
	were the equality $H'_{\Delta_x}(x)=h_{x,x}(x)$ follows from the Fundamental Theorem of Calculus for the Stieltjes derivative (see \cite[Theorem 2.4]{PoRo}).

	\noindent
	{\it Case 2:} Let $M_0(\Delta)$ be the set of all step functions whose discontinuities are not in $D_\Delta$. If $f \in M_0(\Delta)$ we deduce that $F^\Delta=f$ $\mu$--a.e. from Case 1.

	\noindent
	{\it Case 3:} There exists a nondecreasing sequence $(f_n)_{n=1}^{\infty}$ in $M_0(\Delta)$ such that $\lim_{n\to \infty}{f_n(x)}=f(x)$ for $\mu$--a.a. $x \in [a,b)$. We define
	\[F_n(x)=\int_{[a,x)}f_n \dif\mu,\]
	and then it follows from the Lebesgue's Monotone Convergence Theorem for measures that
	\begin{displaymath}F(x)=\lim_{n \to \infty}F_n(x)=F_1(x)+\sum_{k=2}^{\infty}(F_k(x)-F_{k-1}(x))\end{displaymath}
	for all $x \in [a,b]$. Since each summand is a nondecreasing step function of $x$, we can apply Case~2 and Proposition~\ref{fub} to deduce that for $\mu$--a.a. $x \in [a,b)$ we have
	\begin{align*}
	F^\Delta(x) &=F_1^\Delta(x)+\sum_{k=2}^{\infty}(F_k^\Delta(x)-F_{k-1}^\Delta(x)) \\
	&= f_1(x)+\sum_{k=2}^{\infty}(f_k(x)-f_{k-1}(x))=\lim_{n \to \infty}f_n(x)=f(x).
	\end{align*}

	\noindent
	{\it General Case.} For any $f \in {\cal L}^1_\Delta([a,b))$ we have $f=f_1-f_2$, where each of the $f_i$'s is the limit of a nondecreasing sequence of step functions in the conditions of Case 3.
\end{proof}

\begin{dfn}\label{abscont}
Let $x\in\bR$ and $F:[a,b]\to\bR$. We shall say that $F$ is $\Delta_x$--absolutely continuous if for every $\varepsilon>0$, there exists $\delta>0$ such that for every open pairwise disjoint family of subintervals $\{(a_n,b_n)\}_{n=1}^m$ verifying
\begin{displaymath}\sum_{n=1}^m (\Delta_x(b_n)-\Delta_x(a_n))<\delta\end{displaymath}
implies that
\begin{displaymath}\sum_{n=1}^m |F(b_n)-F(a_n)|<\varepsilon.\end{displaymath}
A map $F:[a,b]\to\bR^n$ is $\Delta_x$--absolutely continuous if each of its components is a $\Delta_x$--absolutely continuous function.

\end{dfn}
\begin{rem}
Note that as a consequence of (H4), if $F$ is $\Delta_x$--absolutely continuous, it is $\Delta_y$--absolutely continuous for all $y\in\bR$. Hence, we will just say that $F$ is $\Delta$--absolutely continuous.
\end{rem}

In the following results, we present some of the properties that $\Delta$--absolutely continuous functions share.
\begin{pro}\label{abscomp}
Let $f:[a,b]\to[c,d]$ be a $\Delta$--absolutely continuous function and let $f_2:[c,d]\to\bR$ satisfy a Lipschitz condition on $[c,d]$. The the composition $f_2\circ f_1$ is $\Delta$--absolutely continuous on $[a,b]$.
\end{pro}
\begin{proof}
Let $L>0$ be a Lipschitz constant for $f_2$ on $[c,d]$. Fix $x\in[a,b]$. For each $\e>0$ take $\d>0$ in Definition~\ref{abscont} with $\e$ replaced by $\e/L$. Now, for an open pairwise disjoint family of subintervals $\{(a_n,b_n)\}_{n=1}^m$ such that
\[\sum_{n=1}^m (\Delta_x(b_n)-\Delta_x(a_n))<\delta\]
we have that
\[\sum_{n=1}^m |f_2(f_1(b_n))-f_2(f_1(a_n))|\le L \sum_{n=1}^m |f_1(b_n)-f_1(a_n)|<\varepsilon,\]
that is, $f_2\circ f_1$ is $\Delta_x$--absolutely continuous.
\end{proof}

\begin{pro}
	Let $F:[a,b]\to\bR$ be a $\Delta$--absolutely continuous function. Then $F$ is of bounded variation.
\end{pro}
\begin{proof}
	To prove this result we will use the following remark: if for any $[\alpha,\beta]\subset(a,b)$ there exists $c>0$ such that the total variation of $F$ on $[\alpha,\beta]$ is bounded from above by $c$, then $F$ has bounded variation on $[a,b]$. Indeed, assume that for any $[\alpha,\beta]\subset(a,b)$ there exists $c>0$ such that the total variation of $F$ on $[\alpha,\beta]$ is bounded from above by $c$. Then for each $x\in(a,b)$,
	\[ |F(x)|\le \left|F(x)-F\(\frac{a+b}{2}\)\right|+\left|F\(\frac{a+b}{2}\)\right|\le c+\left|F\(\frac{a+b}{2}\)\right|.\]
	Hence, $|F|$ is bounded on $[a,b]$. Let $K>0$ be one of its bounds. For any partition $\{x_0,x_1,\dots,x_n\}$ of $[a,b]$, we have that
	\[\sum_{k=1}^n |F(x_k)-F(x_{k-1})|=|F(x_1)-F(a)|+\sum_{k=2}^{n-1}|F(x_k)-F(x_{k-1})|+|F(b)-F(x_{n-1})|\le 4K+c,\]
	and so, our claim holds.

	Now, to prove that $F$ has bounded variation on $[a,b]$, fix $x\in[a,b]$ and take $\e=1$ in the definition of $\Delta_x$--absolute continuity. Then, there exists $\delta>0$ such that for any family $\{(a_n,b_n)\}_{n=1}^m$ of pairwise disjoint open subintervals of $[a,b]$,
	\[\sum_{n=1}^m (\Delta_x(b_n)-\Delta_x(a_n))<\delta\implies \sum_{n=1}^m |F(b_n)-F(a_n)|<1.\]
	Consider a partition $\{y_0,y_1,\dots,y_n\}$ of $[\Delta_x(a),\Delta_x(b)]$ such that $0<y_k-y_{k-1}<\delta$, $k=1,2,\dots, n$. Define $I_k=\Delta_x^{-1}([y_{k-1},y_k))$, $k=1,2,\dots, n$. Since $\Delta_x$ is nondecreasing, the sets $I_k$ are empty or they are intervals not necessarily open nor close. Anyway, $[a,b]=\cup I_k$, and so it is enough to show that $F$ has bounded variation on the closure of each $I_k$. We assume the nontrivial case, that is, $\overline I_k=[a_k,b_k]$, $a_k<b_k$. If $[\alpha,\beta]\subset (a_k,b_k)$ and $\{t_0,t_1,\dots,t_m\}$ is a partition of $[\alpha,\beta]$, then
	\[\sum_{i=1}^m \(\Delta_x(t_i)-\Delta_x(t_{i-1})\)=\Delta_x(\beta)-\Delta_x(\alpha)\leq y_k-y_{k-1}<\delta\implies \sum_{i=1}^m |F(x_i)-F(x_{i-1})|<1. \]
	Now, our previous claim implies that $F$ has bounded variation on each $\overline I_k$, and therefore $F$ has bounded variation on $[a,b]$.
\end{proof}	
\begin{pro}
	Let $F:[a,b]\to\bR$ be a $\Delta$--absolutely continuous function. Then $F$ is left--continuous everywhere. Moreover, $F$ is continuous where $\Delta$ is continuous.
\end{pro}
\begin{proof}
	Fix $x\in[a,b]$ and $\e>0$ and let $\delta>0$ be given by the definition of $\Delta_x$--absolute continuity of $F$. Since $\Delta_x(\cdot)$ is left--continuous at $x$, there exists $\delta'>0$ such that if $0<x-t<\delta'$ then, \[\Delta_x(x)-\Delta_x(t)<\delta\implies |F(x)-F(t)|<\e.\]
	The proof in the case $\Delta_x$ is right--continuous at $x\in[a,b)$ is analogous, and we omit it.
\end{proof}

As a consequence of these two previous propositions, given $F$, a $\Delta$--absolutely continuous function, there exist two nondecreasing and left--continuous functions, $F_1,F_2$, such that $F=F_1-F_2$. We denote by $\mu_i:\mathcal B([a,b])\to\bR$ the Lebesgue--Stieltjes measure defined by $F_i$, $i=1,2$. Recall that Lebesgue--Stieltjes measure are positive measures that are also outer regular, that is, for every $E\in\mathcal B([a,b])$, we have
\[\mu_i(E)=\inf\{\mu_i(V): E\subset V, V\mbox{ open}\}. \]
A natural definition for a signed measure for the function $F$ is given by
\begin{displaymath}\mu_F(E)=\mu_1(E)-\mu_2(E),\quad E\in\mathcal B([a,b]).\end{displaymath}
\begin{lem}\label{fabs}
	Let $F:[a,b]\to\bR$ be a $\Delta$--absolutely continuous function. Then for every $x\in[a,b]$ we have $\mu_F\ll\mu.$
\end{lem}
\begin{proof}
	Let $x\in[a,b]$, $\e>0$ and $\delta>0$ given by the definition of $\Delta_x$--absolute continuity with $\e$ replace by $\e/2$. Fix an open set $V\subset (a,b)$ such that $\mu_x(V)<\delta.$ Without loss of generality, we can assume that $V=\bigcup_{n\in\bN}(a_n,b_n)$ for a pairwise disjoint family of open intervals. For each $n\in\bN$, take $a_n'\in(a_n,b_n)$. Then, for each $m\in\bN$ we have
	\[\sum_{n=1}^m(\Delta_x(b_n)-\Delta_x(a_n'))=\mu_{x}\left(\bigcup_{n=1}^m [a_n',b_n) \right)\le \mu_x(V)<\delta,\]
	and so $\sum_{n=1}^m |F(b_n)-F(a_n')|<\e/2.$ By letting $a_n'$ tend to $a_n$, we obtain
	\begin{displaymath}\sum_{n=1}^m |F(b_n)-F(a_n^+)|\le \e/2,\quad \mbox{for each fixed }m\in\bN.\end{displaymath}
	Thus, if $\mu(V)<\delta$ we have that
	\[|\mu_F(V)|=\left|\sum_{n=1}^\infty \mu_F(a_n,b_n) \right|\le \sum_{n=1}^\infty |F(b_n)-F(a_n^+)|<\e. \]

	Let $E\in\mathcal B([a,b])$ be such that $\mu_x(E)=0.$ By outer regularity, there exist open sets $V_n\subset [a,b]$, $n\in \bN$ such that $E\subset V_n$, $n\in\bN$ and
	\begin{displaymath}\lim_{n\to\infty}\mu_x(V_n)=\mu_x(E),\quad \lim_{n\to\infty}\mu_i(V_n)=\mu_i(E),\quad i=1,2.\end{displaymath}
	Now, by the first part of the proof, we know that $\lil_{n\to\infty}\mu_F(V_n)=\mu_F(E)=0$ since $\lil_{n\to\infty}\mu_x(V_n)=\mu_x(E)=0,$ so
	\begin{displaymath}\mu_F(E)=\mu_1(E)-\mu_2(E)=\lim_{n\to\infty}\mu_F(V_n)=0.\end{displaymath}
	Hence, $\mu_F\ll\mu_x$, and since $\mu_x\ll \mu,$ the result follows.
\end{proof}

\begin{lem}\label{intabs}
Let $f\in\mathcal L^1_{\mu}([a,b])$ and consider $F:[a,b]\to\mathbb R$ given by 
\begin{displaymath}F(x)=\int_{[a,x)}f \dif\mu.\end{displaymath}
Then $F$ is $\Delta$--absolutely continuous.
\end{lem}
\begin{proof}
It is enough to consider the case $f\ge 0$, as the general case can be expressed as a difference of two functions of this type. 

Fix $\varepsilon>0$ and $x\in\bR$. Hypothesis (H4, iii) implies that there exists $K>0$ such that $|\gamma(t,x)|<K$ for all $t\in[a,b]$. Since $f\in\mathcal L_{\mu}^1([a,b])=\mathcal L_{\Delta}^1([a,b]),$ there exists $\delta>0$ such that if $E\in\mathcal M$, $\mu_x(E)<\delta$, then $\int_E f\dif\mu_x<\varepsilon/K.$ Now, noting that~\eqref{hbound} holds $\Delta$--almost everywhere, it holds $\mu$--almost everywhere and so
\[\int_E f(s)\dif\mu_x(s)=\int_Ef(s)h_{x,\a(s)}(s)\dif\mu(s)\ge\int_E \frac{f(s)}{\gamma(\a(s),x)}\dif\mu(s)>\frac{1}{K}\int_Ef(s)\dif\mu(s).\]
Thus, if $E\in\mathcal M$, $\mu_x(E)<\delta$, then $\int_E f\dif\mu<\varepsilon.$
 Consider $\{(a_n,b_n)\}_{n=1}^m$ intervals in the conditions of the definition of $\Delta_x $--absolute continuity. Take $E=\cup [a_n,b_n).$ Then
\begin{displaymath}\mu_x(E)=\sum_{n=1}^m \mu_x([a_n,b_n))=\sum_{n=1}^m (\Delta_x (b_n^-)-\Delta_x (a_n^-))=\sum_{n=1}^m (\Delta_x (b_n)-\Delta_x (a_n))<\delta,\end{displaymath}
implies that
\begin{displaymath}\sum_{n=1}^m |F(b_n)-F(a_n)|=\sum_{n=1}^m (F(b_n)-F(a_n))=\sum_{n=1}^m \int_{[a_n,b_n)}f\dif\mu=\int_Ef \dif\mu<\varepsilon.\end{displaymath}
\end{proof}
\begin{thm}\label{FTC2}
	A function $F:[a,b]\to\bR$ is $\Delta$--absolutely continuous on $[a,b]$ if and only if the following conditions are fulfilled:
	\begin{enumerate}
		\item[\textup{(i)}] there exists $F^{\Delta}$ for $\mu$--a.a. $x\in[a,b]$;
		\item[\textup{(ii)}] $F^{\Delta}\in \mathcal L^1_{\mu}([a,b))$;
		\item[\textup{(iii)}] for each $x\in[a,b]$, 
		\[F(x)=F(a)+\int_{[a,x)}F^{\Delta}\dif\mu
		.\]
	\end{enumerate}
\end{thm}
\begin{proof}
	Lemma~\ref{intabs} ensures that the three conditions are sufficient for $F$ to be $\Delta$--absolutely continuous. For the converse, consider $\mu_F$ to be the Lebesgue--Stieltjes measure defined by $F$ and let $z\in[a,b]$ be fixed. Lemma~\ref{fabs} and the Radon--Nykodym Theorem guarantee that there exists a measurable function $l:([a,b),\cB)\to(\bR,\cB)$ such that
	\[\mu_F(E)=\int_{E}l\dif\mu,\quad \mbox{for any Borel set }E\subset[a,b).\]
	In particular, 
	\[F(x)-F(a)=\mu_F([a,x))=\int_{[a,x)}l(s)\dif\mu(s).\]
	Theorem~\ref{FTC} ensures that $F^{\Delta}(s)=l(s)$ for $\mu$--a.a. $s\in[a,b)$, and so the result follows.
\end{proof}


\section{The relation with Stieltjes derivatives}\label{Stieltjes}
As commented before, Stieltjes derivatives are, in a first approach, a particular case of $\Delta$--derivatives. However, there can be proven to be equivalent to the displacement derivatives if hypotheses (H1)--(H5) hold. Indeed, we shall prove this equivalence through following results.

\begin{pro}\label{equivgdelt1}
	Let $\Delta:[a,b]^2\to\mathbb R$ satisfy \textup{(H1)--(H5)} and $g:[a,b]\to\mathbb R$ be as in~\eqref{gfunc}. Then,
	\[
	\lim_{s\to t}\frac{g(s)-g(t)}{\Delta(t,s)}=1,\quad\mbox{for all }t\in [a,b]\setminus C_\Delta.
	\]
\end{pro}
\begin{proof}
	Fix $t\in[a,b]\setminus C_\Delta$. 
	It follows from~\eqref{hboundz} that
	$1/\gamma(t,r)\le h_{r,t}(r)\le \gamma(r,t)$ for all $r\in[a,b]$.
	As a consequence, we have that
	\begin{equation}\label{hineq1}
	\frac{1}{\displaystyle\sup_{r\in[t,s)}\gamma(t,r)}\le h_{r,t}(r)\le \sup_{r\in[t,s)}\gamma(r,t).
	\end{equation}

	Since $t\not\in C_\Delta$ we have that 
	\begin{equation}\label{conddelta1}
	\Delta_x(y)>0\ \ \mbox{ for all }y>x,\quad\mbox{and/or}\quad \Delta_x(y)<0\ \ \mbox{ for all }y<x.
	\end{equation}

	Assume the first case in~\eqref{conddelta1} holds. Hence, for any $s\in[a,b]$, $s>t$, we have that
	\begin{equation*}\label{gdeltalim1}
	\frac{g(s)-g(t)}{\Delta(t,s)}=\frac{\int_{[a,s)} h_{r,t}(r)\dif\mu_t-\int_{[a,t)}h_{r,t}(r)\dif\mu_t}{\Delta(t,s)}=\frac{\int_{[t,s)}h_{r,t}(r)\dif\mu_t}{\Delta(t,s)}.
	\end{equation*}
	Now, for any $s>t$ it follows from~\eqref{hineq1}, that
	\[
	\frac{1}{\displaystyle\sup_{r\in[t,s)}\gamma(t,r)}\Delta(t,s)
	\le \int_{[s,t)}h_{r,t}(r)\dif\mu_t
	\le 
	\sup_{r\in[t,s)}\gamma(r,t)\Delta(t,s),
	\]
	since $\mu_t([t,s))=\Delta(t,s)$. Equivalently,
	\[\frac{1}{\displaystyle\sup_{r\in[t,s)}\gamma(t,r)}
	\le \frac{\int_{[t,s)}h_{r,t}(r)\dif\mu_t}{\Delta(t,s)}=\frac{g(s)-g(t)}{\Delta(t,s)}
	\le 
	\sup_{r\in[t,s)}\gamma(r,t),\quad s>t,\]
	and so, allowing $s\to t^+$, we obtain
	\[
	1=\lim_{s\to t^+}\frac{1}{\displaystyle\sup_{r\in[t,s)}\gamma(t,r)}
	\le \lim_{s\to t^+}\frac{g(s)-g(t)}{\Delta(t,s)}
	\le 
	\lim_{s\to t^+}\sup_{r\in[t,s)}\gamma(r,t)=1.\]
	If $\Delta_x(\cdot)=0$ on some $[x-\delta,x]$, $\delta >0$, then the proof is complete. 
	Otherwise, the second condition in~\eqref{conddelta1} holds. In that case we have that for any $s\in[a,b]$, $s<t$,
	\begin{equation}\label{gdeltalim2}
	\frac{g(s)-g(t)}{\Delta(t,s)}=\frac{\int_{[a,s)} h_{r,t}(r)\dif\mu_t-\int_{[a,t)}h_{r,t}(r)\dif\mu_t}{\Delta(t,s)}=-\frac{\int_{[s,t)}h_{r,t}(r)\dif\mu_t}{\Delta(t,s)}.
	\end{equation}
	Once again, it follows from~\eqref{hineq1}, that for any $s<t$,
	\[
	-\sup_{r\in[t,s)}\gamma(r,t)(-\Delta(t,s))
	\le -\int_{[s,t)}h_{r,t}(r)\dif\mu_t
	\le -\frac{1}{\displaystyle\sup_{r\in[t,s)}\gamma(t,r)}(-\Delta(t,s)) ,
	\]
	since $\mu_t([s,t))=-\Delta(t,s)$. Equivalently,
	\[\sup_{r\in[t,s)}\gamma(r,t)
	\le -\frac{\int_{[s,t)}h_{r,t}(r)\dif\mu_t}{\Delta(t,s)}=\frac{g(s)-g(t)}{\Delta(t,s)}
	\le 
	\frac{1}{\displaystyle\sup_{r\in[t,s)}\gamma(t,r)},\quad s<t.\]
	Now the rest of the proof is analogous to the previous case, and we omit it.
\end{proof}

Bearing in mind the following relation,
\[\lim_{s\to t}\frac{f(s)-f(t)}{\Delta(t,s)}=\lim_{s\to t}\frac{f(s)-f(t)}{g(s)-g(t)}\frac{g(s)-g(t)}{\Delta(t,s)},\quad\mbox{for all }t\in [a,b]\setminus C_\Delta,\] 
we obtain the next result.

\begin{thm}\label{thmeq}
	Let $\Delta:[a,b]^2\to\mathbb R$ satisfy \textup{(H1)--(H5)} and $g:[a,b]\to\mathbb R$ be given by~\eqref{gfunc}. Given $f:[a,b]\to\mathbb R$ and $t\in [a,b]\setminus O_\Delta$, $f^\Delta(t)$ exists if and only if $f'_g(t)$ exists.
\end{thm}

Therefore, we have that both derivatives are, indeed, equivalent. This shows the interest of studying this type of derivatives. In particular, when looking at differential equations, a wide variety of results exists in several papers such as \cite{PoRo} where the authors showed that Stieltjes differential equations are good for studying equations on time scales and impulsive differential equations, or \cite{FP2016,PoMa,LopezPouso2019,LopezPouso2019a, LopezPouso2018} where we can find different types of existence and uniqueness of solution results.

However, obtaining the corresponding function $g$ can be hard. Indeed, although~\eqref{gfunc} gives an explicit expression of the function, it is defined in terms of the $\Delta$--measure, which depends on the functions~\eqref{hzzdef} given by the Radon--Nidok\'ym Theorem. The following result gives a simple way to obtain, under certain hypotheses, the corresponding function $g$.

\begin{pro}\label{propgint}
	Let $\Delta:[a,b]^2\to\mathbb R$ satisfy \textup{(H1)--(H5)} and $g:[a,b]\to\mathbb R$ be given by~\eqref{gfunc}. Suppose that $D_2\Delta$ exists, is continuous and positive on $[a,b]^2$. Then $g'(t)$ exists for all $t\in[a,b]$, and 
		\[g'(t)=D_2\Delta(t,t).\]
		In particular, $g$ is strictly nondecreasing and $g\in\mathcal C^1([a,b])$.
\end{pro}
\begin{proof}
	Fix $t\in[a,b]$. Since $D_2\Delta(t,t)>0$, the function $\Delta_t$ is strictly nondecreasing in a neighborhood of $t$. Thus, $t\not\in C_\Delta$, and so, if we apply Proposition~\ref{equivgdelt1}, we have that
	\[D_2\Delta(t,t)=\lim_{s\to t}\frac{\Delta(t,s)}{s-t}=\lim_{s\to t}\frac{\Delta(t,s)}{s-t}\lim_{s\to t}\frac{g(s)-g(t)}{\Delta(t,s)}=\lim_{s\to t}\frac{g(s)-g(t)}{s-t}.\]
	Hence, $g'(t)$ exists and equals to $D_2\Delta(t,t)$. The rest of the result now follows.
\end{proof}
\begin{exa}
	Consider the non--Stieltjes displacement $\Delta:[0,1]\times[0,1]\to\mathbb R$ in Example~\ref{ExDeltanoSti}, given by
	\[\Delta(x,y)=e^{y^2-x^2}-e^{x-y}.\]
	It follows from~\eqref{nonstider} that the hypotheses of Proposition~\ref{propgint} are satisfied. Hence, we can compute $g:[0,1]\to\mathbb R$ in~\eqref{gfunc} as
	\[g(t)=\int_0^t D_2\Delta(s,s)\dif s=\int_0^t (2s+1)\dif s=t^2+t.\]
\end{exa}

\section{A model for smart surface textures}

In this section we develop a model for smart surfaces based on a displacement and a derivation similar to that of the diffusion equation.

It is every day more extended to employ biomimetics in order to develop surfaces with extraordinary properties \cite{Yuan2019}. These meta-materials imitate organic tissues with special microstructures which modify their usual behavior. For instance, a cats tongue possesses backwards-facing spines, a disposition which facilitates that particles move towards the interior of the mouth and not in the other direction. This situation is similar to the one on the human respiratory epithelium, with the difference that in this other tissue the effect is due to the active motion of the cells cilia, which move mucus and particles upwards, and not a passive result of the microstructure.

With 3D-printing (or other methods) we may obtain these surfaces for which friction depends on direction, position, pressure, etc. We can measure friction in an indirect but simple way using the definition of work: \emph{work is the energy necessary to move and object between two points against a given force field}. Thus, on our surface, which, for convenience in the present discussion, we will consider one-dimensional (the higher dimensional case would be analogous), we can define a function $W(x,y)$ which measures the work necessary to move a point mass from $x$ to $y$.

If we consider now a distribution of particles on the surface subject to random vibrations, this kind of situation may be described as a Brownian motion but, in the case those particles are very small, this model may be approximated by a diffusion process. The derivation of the classical diffusion process --that is, of Fick's second law or, equivalently, the heat equation-- can be found in many references --see for instance \cite{Kirkwood2018,Tikhonov1990,Haberman1998}.

By Nernst's law, in the friction-less setting, the mass flowing through the point $x$ during the time interval $(t,t+\dif t)$ is equal to
\[Q=-D\frac{\Delta u(x,t)}{\dif x}\dif t,\] where $u$ denotes the concentration of mass, $\Delta u$ the spatial variation of $u$, $D$ is the diffusion coefficient and $\dif x$, $\dif t$ are considered to be \emph{infinitesimal quantities}. With friction, this variation of the mass flow is impeded by the work necessary to move the particles, that is, $W$. We will assume that the variation of the mass flow is inversely proportional to the spatial variation of this work, that is,
\[Q=-D\frac{\Delta u(x,t)}{W(x,x+\dif x)}\dif t.\]
On the other hand, the spatial variation of $Q$ can be computed directly as
\[\Delta Q=Q(x,t)-Q(x+\dif x,t)+h(x)\dif x\dif t,\]
where $h$ is a source term in the case we allow for a continuous inflow of particles. At the same time, the variation in the mass flowing through the sectional volume, being proportional to the concentration of mass, can be computed as
\[\Delta Q=\left[-D\frac{\Delta u(x,t)}{W(x,x+\dif x)}+D\frac{\Delta u(x+\dif x,t)}{W(x+\dif x,x+2\dif x)}+h(x)\dif x\right]\dif t=c\Delta u(x,t)\dif x.\] for some constant $c$.
In the limit $\dif x\to 0$,
\[\frac{\Delta u(x,t)}{W(x,x+\dif x)}\approx \partial_{x,W}\Delta u(x,t),\quad \frac{\Delta u(x+\dif x,t)}{W(x+\dif x,x+2\dif x)}\approx \partial_{x,W}\Delta u(x+\dif x,t).\]
where $\partial_{x,W}$ denotes the displacement derivative for the displacement $W$ with respect to the variable $x$.
Thus,
\[\left[-D\partial_{x,W}\Delta u(x,t)+D\partial_{x,W}\Delta u(x+\dif x,t)+h(x)\dif x\right]\dif t=c\Delta u(x,t)\dif x,\]
that is,
\[D\left[\frac{\partial_{x,W}\Delta u(x+\dif x,t)-\partial_{x,W}\Delta u(x,t)}{\dif x}\right]+h(x) =c\frac{\Delta u(x,t)}{\dif t}.\]
Again, in the limit,
\[D\partial_x\partial_{x,W} u(x,t)+h(x)=c\partial_t u(x,t).\]

Now we consider, for instance, the stationary problem with mixed two-point boundary conditions
\begin{equation}\label{2pb} u'_W\,'(x)+h(x)=0,\quad u'_W(0)=0,\ u(1)=C.\end{equation}
where $C$ is a real constant. Integrating, we obtain the problem
\begin{equation}\label{2pb2} u'_W(x)=-\int_0^xh(y)\dif y,\quad u(1)=C.\end{equation}

 Under the conditions of Theorem~\eqref{thmeq} --that is, (H1)--(H5)-- we can apply theorems such as \cite[Theorem~3.5]{PoMa} to derive the existence of solution of problem~\eqref{2pb2}.

\section{Conclusions}

In the work behind we have established a theory of Calculus based on the concept of displacement. We have studied the associated topology and measure and proved some general results regarding their interaction. We have also defined new concepts such as the integral with respect to a path of measures or the displacement derivative, studied their properties and proved a Fundamental Theorem of Calculus that relates them. Finally, we have set up a framework in order to study displacement equations. We have proved they can be transformed into Stieltjes differential equations and so the results in \cite{FP2016,PoMa} can be applied to new mathematical models.

We have also left some open problems in our way. First of all, conditions (H1) and (H2) may be weakened further, of substituted by a different set of axioms in order to achieve the same results. It will also be interesting to analyze how these relate to other concepts that generalize the notion of metric space, such as those derived from the conditions in \cite{Lorenz,Roldan}, and to further explore the topological properties of displacement spaces.

To explore how to weaken conditions (H3)--(H5) would be an even more important task. Although very general in nature, they bound displacements to Stieltjes derivatives in a stringent way, such as is shown in Proposition~\ref{plc}. In the same way, it would be interesting to generalize the theory of displacement derivatives to the case where a general displacement is also consider in the numerator. We believe that the natural way to define the derivative in that case is as presented in the following definition.
\begin{dfn}\label{deriv2} 
	Let $([a,b],\Delta_1)$ satisfy \textup{(H1)--(H5)} and $(\bR,\Delta_2)$ satisfy \textup{(H1)--(H3)}.
	The \emph{derivative with respect to the pair of displacements $(\Delta_1,\Delta_2)$} (or \emph{$\Delta_1^2$--derivative}) of a function $f:[a,b]\to \bR$ at a point $x\in [a,b]\setminus O_{\Delta_1}$ is defined as follows, provided that the corresponding limits exist:
	 \begin{align*} f^{\Delta_1^2}(x) =\begin{dcases}\lim_{y \to x}\frac{\Delta_2(f(x),f(y))}{\Delta_1(x,y)}, & x\not\in D_{\Delta_1},\\
 \lim_{y \to x^+}\frac{\Delta_2(f(x),f(y))}{\Delta_1(x,y)}, & x\in D_{\Delta_1}.\end{dcases}
 \end{align*}
\end{dfn} 
Again, this definition would establish a more general setting than Stieltjes derivatives, but would also include absolute derivatives and some very well known operators, such as the $\phi$--Laplacian~\cite{Cabada2014b}. If we consider two functions $f,\phi:\bR\to\bR$, the $\phi$--Laplacian of $y$ is given by
\[(\phi\circ f')'=(f')^{\Delta_u^2}=(f^{\Delta_u^u})^{\Delta_u^2},\]
where $\Delta_2(x,y)=\phi(y)-\phi( x)$ for $x,y\in\bR$.
Note, however, that since displacements need not to be linear in any sense, this definition could make it more difficult to prove a Fundamental Theorem of Calculus.

We have also hinted at the possibility of developing a theory non symmetric length spaces, based on the concept of displacement, that generalizes the results in \cite{Burago}. Furthermore, fixed point theorems in both displacement spaces and vector displacement spaces should be studied, focusing, in the last case, on the compatibility of the displacement with the underlying order topology (cf. \cite{Nieto,Roldan}). 


Last, the properties of the integral with respect to a path of measures have to be thoroughly studied, as this concept lays many possibilities ahead.

\end{document}